\documentclass[leqno,10.5pt]{article} 
\usepackage{float}
\usepackage{hyperref} 
\usepackage{amsmath, amssymb}
\usepackage{amsthm} 
\usepackage{amscd}
\usepackage{color}
\usepackage[normalem]{ulem}

\usepackage{verbatim} 
\usepackage{amsmath, amssymb}
\usepackage{amsthm} 
\usepackage{amscd}
\usepackage{multicol,lineno}%
\usepackage{enumitem}
\theoremstyle{plain} 
\newtheorem{theorem}{\indent\sc Theorem}[section]
\newtheorem{lemma}[theorem]{\indent\sc Lemma}
\newtheorem{corollary}[theorem]{\indent\sc Corollary}
\newtheorem{proposition}[theorem]{\indent\sc Proposition}
\theoremstyle{definition} 
\newtheorem{definition}{\indent\sc Definition}
\newtheorem{remark}[theorem]{\indent\sc remark}
\newtheorem{example}[theorem]{\indent\sc Example}

\newcommand{\C}{\mathbb{C}}
\newcommand{\R}{\mathbb{R}}
\newcommand{\Q}{\mathbb{Q}}
\newcommand{\Z}{\mathbb{Z}}
\newcommand{\N}{\mathbb{N}}

\newcommand{\Li}{\rm{Li}}

\newcommand{\AND}{\quad\textrm{and}\quad}

\newcommand{\btheta}{\boldsymbol{\theta}}
\newcommand{\bx}{\textbf{x}}
\newcommand{\by}{\textbf{y}}
\newcommand{\cC}{{\mathcal{C}}}

\newcommand\ee{\varepsilon}
\newcommand{\GL}{\mathrm{GL}}
\newcommand\GrO{\mathcal{O}} 
\newcommand{\Mat}{\mathrm{Mat}}
\newcommand{\norm}[1]{\|#1\|}

\newcommand\pu{\underline{\psi}}

\def \hfillx {\hspace*{ -\textwidth} \hfill}


\title{Pad\'e approximation for a class of hypergeometric functions \\and parametric geometry of numbers}

\author{\textsc{Makoto Kawashima} and \textsc{Anthony Po\"{e}ls}}
\date{2022 March 1} 

\newcommand{\keywords}{{
  \footnotesize
  \textbf{Keywords}: Pad\'e approximation, irrationality exponent, hypergeometric functions, effective Poincar\'e-Perron theorem, parametric geometry of numbers.
}}

\begin{document}

\maketitle

\begin{abstract}
    In this article we obtain new irrationality measures for values of functions which belong to a certain class of hypergeometric functions including shifted logarithmic functions, binomial functions and shifted exponential functions. We explicitly construct Pad\'e approximations by using a formal method and show that the associated sequences satisfy a Poincar\'e-type recurrence. To study precisely the asymptotic behavior of those sequences, we establish an \emph{effective} version of the Poincar\'e-Perron theorem. As a consequence we obtain, among others, effective irrationality measures for values of binomial functions at rational numbers, which might have useful arithmetic applications. A general theorem on simultaneous rational approximations that we need is proven by using new arguments relying on parametric geometry of numbers.
\end{abstract}


\section{Introduction}\label{section: intro}

Let $\theta\in\R\setminus\Q$. We say that a non-negative real number $\mu$ is an \textsl{irrationality measure} of $\theta$ if there exist positive constants $c$, $q_0$ such that
\begin{align}\label{eq: def irr exponent}
    \left\vert\theta-\dfrac{p}{q}\right\vert\geq \dfrac{c}{q^{\mu}}
\end{align}
for any rational numbers $p/q$ with $q \geq q_0$. The irrationality exponent $\mu(\theta)$ of $\theta$ is defined as the infimum of the set of its irrationality measures (with the convention that $\mu(\theta)=\infty$ if this set is empty). The theory of continued fractions shows that $\mu(\theta)\geq 2$, with equality for almost all $\theta$ in the sense of the Lebesque measure. On the other hand, the theorem of K.~F.~Roth implies that we also have $\mu(\theta)=2$ if $\theta$ is an irrational algebraic number. However, it is well-known that this result is non-effective; for a given $\mu>2$ close to $2$, we do not know how to compute the constants $c$ and $q_0$ satisfying \eqref{eq: def irr exponent}.

\medskip

One of the prominent points in the theory of Pad\'e approximation is to produce effective or explicit estimates, as it is crucial for specific applications to irrationality questions and Diophantine equations. See for example the results of A.~Baker \cite{ABaker}, M.~A.~Bennett \cite{BennettSimul, BennettAus, BennettCrelle}, D.~V.~Chudnovsky and G.~V.~Chudnovsky \cite{ch9, Chubrothers84, ch11}, which give sharp irrationality measures  related to binomial functions and refined bounds for the number of the solutions to certain Diophantine equations, as well as the remarkable study on Pellian equations over function fields by U.~Zannier \cite{Zannier}. The study of values of functions such as $G$-functions and $E$-functions is of particular interest, see for example the book of N.~I.~Fel'dman and  Y.~V.~Nesterenko \cite{FelNest1998} which contains most references related to the subject before $1998$. There are also several related works by S.~Fischler and T.~Rivoal \cite{FischlerRivoal2014, FischlerRivoal}, A.~I.~Galochkin \cite{G1, G2, G3}, K.~V$\ddot{\text{a}}$$\ddot{\text{a}}$n$\ddot{\text{a}}$nen \cite{Va, Va2}, P.~Voutier \cite{Voutier} and W. Zudilin \cite{Z}, which establish linear independence measures for values of certain hypergeometric functions. Recent results on polylogarithms and generalized hypergeometric $G$-functions were obtained by S.~David, N.~Hirata-Kohno and the first author of this article in \cite{DHK2,DHK3,DHK4,DHK5}.

\medskip

Let $\gamma,\omega,x$ be rational numbers with $\gamma\notin\{-1,-2,\dots\}$, $\omega\notin\Z$ and $x\in\mathbb{Q}\cap[0,1)$. We give a sufficient condition on $\beta\in\Q$ under which $\theta=f(\beta)$ is irrational, and we then estimate its irrationality exponent, where $f(z)$ is a hypergeometric function of the form
\begin{equation}\label{eq: three special cases for f}
    f(z)=
    \begin{cases}
        {\displaystyle{\sum_{k=0}^{\infty}}}\dfrac{(-\omega)_k}{k!}\dfrac{1}{z^{k+1}}=\dfrac{1}{z}\cdot{}_{2}F_{1} \biggl(\begin{matrix}-\omega, 1\\1 \end{matrix} \biggm|~\dfrac{1}{z}\biggr)=\dfrac{1}{z}\left(1-\dfrac{1}{z}\right)^{\omega}& \text{(binomial~function)}\\
        {{(1+x)}\displaystyle{\sum_{k=0}^{\infty}}}\dfrac{1}{(k+x+1)}\dfrac{1}{z^{k+1}}=(1+x)\Phi_1(x,1/z)& \text{(shifted logarithmic function)}\\
        {\displaystyle{\sum_{k=0}^{\infty}}}\dfrac{1}{(\gamma+2)_k}\dfrac{1}{z^{k+1}}=\dfrac{1}{z}\cdot{}_{1}F_{1} \biggl(\begin{matrix} 1~\\ \gamma+2 \end{matrix} \biggm|~\dfrac{1}{z}\biggr)=:\exp_{\gamma}(1/z)
        & \text{(shifted exponential function)}.
    \end{cases}
\end{equation}
Here, $\Phi_s(x,z)$ is the $s$-th Lerch function (see next section for the precise definition). For given positive integers $p,q$ and non-zero complex numbers $a_1, \ldots, a_p, b_1, \ldots, b_{q}$ (where $b_1,\ldots b_q$ are not negative integers), the function ${}_{p}F_{q}$ denotes the usual generalized hypergeometric function
\begin{eqnarray*}
    {}_{p}F_{q} \biggl(\begin{matrix} a_1,\ldots, a_p ~\\ b_1,\ldots, b_{q} \end{matrix} \biggm|~\dfrac{1}{z}\biggr)
    =\displaystyle\sum_{k=0}^{\infty}\dfrac{(a_1)_k \cdots (a_{p})_k}{(b_1)_k\cdots(b_{q})_k}\dfrac{1}{k!\cdot z^k}\enspace,
\end{eqnarray*}
where $(a)_k$ is the $k$-th Pochhammer symbol $(a)_k= a(a+1)\cdots(a+k-1)$ (with the convention $(a)_0 = 1$). It is worth mentioning that the set of functions given by \eqref{eq: three special cases for f} contains both $G$-functions and $E$-functions. More generally, we study the irrationality of the value $f(\beta)$ for any hypergeometric function $f(z)$ defined by
\begin{equation*}
    f(z)=\sum_{k=0}^{\infty}\dfrac{\prod_{i=1}^k(\alpha i-\delta)}{(\gamma+2)_k}\dfrac{1}{z^{k+1}}\enspace,
\end{equation*}
with $\alpha, \beta, \gamma, \delta\in \Q$, $\gamma\ge -1$, and where $\beta$ is a ``large'' compared to its denominator. By choosing our parameters $(\alpha,\gamma,\delta)$ respectively of the form $(1,-1,1+\omega)$, $(1,x,-x)$ and $(0,\gamma,-1)$, we obtain the examples \eqref{eq: three special cases for f}. Our aim is also to provide new tools to approach irrationality problems in general.

\medskip

In a second important and self-contained part of our article, we establish \emph{an effective version} of the Poincar\'e-Perron theorem, which applies in particular when the coefficients of
the Poincar\'e-type recurrence are rational functions. This is needed to obtain effective irrationality measures of $\theta=f(\beta)$ as above. It might also be useful for other arithmetic problems, such as the study of the solutions of Diophantine equations. Note that K.~Alladi and M.~L.~Robinson \cite{A-R} as well as Chudnovsky \cite{ch9} used Poincar\'e-Perron theorem in their work. However, our effective version gives relatively more precise asymptotic estimates than those found in the literature. Also note that the Poincar\'e-Perron theorem is \emph{not} used in \cite{DHK2,DHK3,DHK4,DHK5}, although our method to construct Pad\'e approximants is similar to theirs.

\medskip

Our strategy can be summarized as follows. Regarding $f(z)$ as a formal Laurent series, we first construct an \emph{explicit} sequence of Pad\'e approximants $\big(P_{n,0}(z),P_{n,1}(z)\big)_{n\geq 0}$ for $f(z)$, as in \cite{DHK2,DHK3,DHK4,DHK5}.

\medskip

We then show that the sequences of polynomials $(P_{n,0}(z))_{n\geq 0}$ and $(P_{n,1}(z))_{n\geq 0}$ satisfy a certain Poincar\'e-type recurrence of order $2$, which allows us, thanks to the Poincar\'e-Perron theorem, to control the asymptotic behavior of the quantities $P_{n,0}(\beta)$, $P_{n,1}(\beta)$ and $R_n(\beta):=P_{n,0}(\beta)f(\beta)-P_{n,1}(\beta)$ as $n$ tends to infinity for any large enough $\beta\in\Q$.

\medskip

Lastly, we control the size of the denominators of the rational numbers $P_{n,0}(\beta)$ and $P_{n,1}(\beta)$. For that purpose, we estimate precisely the quotient of Pochhammer symbols in Section~\ref{quotient}. To complete our proof, we then apply a general theorem on simultaneous rational approximations, whose proof is based on a new argument relying on parametric geometry of numbers. This leads us to several new irrationality measures for $f(\beta)$. In the case of binomial functions, we extend our result to a more general setting by taking $\beta$ in a number field $K$ and by replacing the usual absolute value by the one coming from a given place of $K$. In particular, if the place is non-archimedean, we obtain a $p$-adic version of our theorem for binomial functions.

\medskip

This article is organized as follows. Our main result is stated in Section~\ref{notations}. In Section~\ref{padeformal}, we construct Pad\'e approximants $\big(P_{n,0}(z),P_{n,1}(z)\big)_{n\geq 0}$ as in  \cite{DHK2,DHK3,DHK4,DHK5} and show that they satisfy a certain Poincar\'e-type recurrence. The asymptotic estimates for $|P_{n,0}(\beta)|$, $|P_{n,1}(\beta)|$, $|R_n(\beta)|$ and for the denominators of $P_{n,0}(\beta)$ and $P_{n,1}(\beta)$ are established in Section~\ref{quotient}, modifying and generalizing estimates for the denominators of quotients of Pochhammer symbols from \cite[Lemma $10$]{Lepetit}. Our main theorem is proved is Section~\ref{mainproof}. The last argument of the proof is a consequence of a general result on simultaneous approximations in Section~\ref{parametricgn} relying on parametric geometry of numbers.
Section~\ref{section: Poincare-Perron thm} is devoted to the proof of an effective version of the Poincar\'e-Perron theorem, leading to more precise effective estimates for $|P_{n,0}(\beta)|$, $|P_{n,1}(\beta)|$ and $|R_n(\beta)|$. This is an important key tool to obtain effective irrationality measures.
Explicit examples of irrationality measures for cubic roots are given in Section~\ref{example}, where we compare our results with previous ones. Finally, we give a general statement on binomial functions, which covers the $p$-adic case in Section~\ref{bin}.

\section{Notation and main result}\label{notations}

We denote by $\N$ the set of strictly positive rational integers. Let $z\in \C$ with $|z|<1$. Consider a positive integer $s$ and $x\in \Q$ with $0 \le x <1$. The $s$-th Lerch function, which is a generalized polylogarithmic function, is defined by
\[
    \Phi_s(x,z) =\displaystyle\sum_{k=0}^{\infty}\frac{z^{k+1}}{{(k+x+1)}^s}\enspace.
\]
The function $\Phi_1(x,z)$ is called the shifted logarithmic function with shift $x$, and the function $\Phi_s(0,z)$ is the polylogarithmic function $\Li_s$ with depth $s$. These functions converge in $\vert z \vert<1$. Let $\alpha, \gamma, \delta\in \Q$ be parameters with $\gamma\ge -1$, and define the Laurent series $f(z)$ by
\begin{equation}\label{our function}
    f(z)=\sum_{k=0}^{\infty}\dfrac{\prod_{i=1}^k(\alpha i-\delta)}{(\gamma+2)_k}\dfrac{1}{z^{k+1}}\enspace,
\end{equation}
with the convention $\prod_{i=1}^k(\alpha i-\delta)=1$ if $k=0$. We will develop $f(z)$ at the infinity for  $\vert z \vert> |\alpha|$. For a given $z\in \C$ with $|z|>|\alpha|$, we denote by
\[
    \rho_1(\alpha,z) \leq \rho_2(\alpha,z)
\]
the moduli of the roots $2z-\alpha \pm 2\sqrt{z^2-z\alpha}$ of the characteristic polynomial
\begin{equation*} 
    P(X) = X^2-2(2z-\alpha)X+\alpha^2.
\end{equation*}
The condition $|z|>|\alpha|$ implies that $\rho_1(\alpha,z) \neq \rho_2(\alpha,z)$ ({\it see} Lemma \ref{apply P-P}). Given a non-empty finite set of algebraic numbers $S$, we put
\begin{align*}
&{\rm{den}}(S)=\min\{1\leq n\in \Z \mid  \text{$n\alpha$ is an algebraic integer for each $\alpha\in S$}\}.
\end{align*}
Let $n$ be a non-negative integer and $y\in \Q$. We define
\begin{align*}
    \nu(y)=\prod_{\substack{q:\rm{prime} \\ q|{\rm{den}}(y)}}q^{q/(q-1)} \AND \nu_n(y)= \prod_{\substack{q:\rm{prime} \\ q|{\rm{den}}(y)}}q^{n+\lfloor n/(q-1)\rfloor}.
\end{align*}
Note that $\nu(y+m) = \nu(y)$ and $\nu_n(y+m) = \nu_n(y)$ for any integer $m\in\Z$. We are now ready to state our main result.

\begin{theorem}[Main Theorem]\label{main}
Let $\alpha,\beta,\gamma,\delta\in \Q$ with $|\beta|>|\alpha|$, $\gamma\ge -1$, $\delta\notin \alpha \N$ and $-(\alpha\gamma+\delta)\notin \alpha \N$. Suppose $\alpha\neq 0$ and define $\Delta$, $E$ and $Q$ by
\begin{align*}
    \Delta & = {\rm{den}}(\alpha)\cdot{\rm{den}}(\beta)\cdot{\rm{exp}}\left(\dfrac{{\rm{den}}(\gamma)}{\varphi({\rm{den}}(\gamma))}\right)\cdot \nu(\gamma) \cdot \nu(\delta/\alpha), \\
         Q & =\rho_2(\alpha,\beta)\cdot\Delta, \\
         E & = \big(\rho_1(\alpha,\beta)\cdot\Delta\big)^{-1},
\end{align*}
where $\varphi$ denotes the Euler's totient function. Assume $E>1$. Then the real number $f(\beta)$ is irrational, and its irrationality exponent satisfies
\[
    \mu(f(\beta))\le 1+\dfrac{\log(Q)}{\log(E)}.
\]
If $\alpha = 0$, then $f(\beta)$ is irrational and $\mu(f(\beta)) = 2$.
\end{theorem}

Note that in the case $\alpha\neq 0$, it is possible to obtain effective irrationality measures in our theorem (see Remark~\ref{remark: remark effective measures} for more details). When $f(z)$ is one of the special functions of \eqref{eq: three special cases for f}, we have the following results. Note that $\alpha=1$ implies that the product of the two roots of the characteristic polynomial $P$ is equal to $1$, so that $\rho_1(1,\beta)=\rho_2(1,\beta)^{-1}$.

\begin{corollary}[shifted logarithmic function]\label{log shift}
Let $\beta,x \in \Q$ with $|\beta|>1$ and $0\leq x < 1$. Define
\begin{align*}
    &\Delta= {\rm{den}}(\beta)\cdot{\rm{exp}}\left(\dfrac{{\rm{den}}(x)}{\varphi({\rm{den}}(x))}\right)\cdot \nu(x),\\
    &Q=\rho_2(1,\beta)\cdot\Delta,\\
    &E=\rho_2(1,\beta)\cdot\Delta^{-1},
\end{align*}
and assume $E>1$. Then $\Phi_1(x,1/\beta)$ is irrational, and its irrationality exponent satisfies
\[
    \mu(\Phi_1(x,1/\beta))\le 1+\dfrac{\log(Q)}{\log(E)}\,.
\]
\end{corollary}

\begin{corollary}[binomial function]\label{binom}
Let $\omega\in \Q\setminus\Z$ and $\beta\in \Q$ with $|\beta|>1$.
Given a non-negative integer $n$, we put
\[
    G_n(\omega)={\rm{GCD}}\left(\nu_n(\omega)\binom{n+k-1}{k}\binom{n-\omega-1}{n-k}, \ \nu_n(\omega)\binom{n+k'}{k'}\binom{n+\omega}{n-1-k'}\right)_{\substack{0\le k \le n\\ 0\le k'\le n-1}},
\]
where ${\rm{GCD}}$ means the greatest common divisor.
Define
\begin{align*}
&\Delta=\Delta(\omega,\beta)=\nu(\omega)\cdot{\rm{den}}(\beta)\cdot\limsup_{n\to \infty}G_n(\omega)^{-1/n},\\
&Q=\rho_2(1,\beta)\cdot \Delta,\\
&E=\rho_2(1,\beta)\cdot \Delta^{-1},
\end{align*}
and assume $E>1$. Then the real number $(1-1/\beta)^{\omega}$ is irrational, and its irrationality exponent satisfies
\[
    \mu((1-1/\beta)^{\omega})\le 1+\dfrac{\log(Q)}{\log(E)}.
\]
In particular, since $\Delta \le \nu(\omega){\rm{den}}(\beta)$, we have
\[
    \mu((1-1/\beta)^{\omega})\le 1+\dfrac{\log\rho_2(1,\beta)+\log \nu(\omega)+\log {\rm{den}}(\beta)}{\log\rho_2(1,\beta)-\log \nu(\omega)-\log{\rm{den}}(\beta)}.
\]
\end{corollary}

\begin{corollary}[shifted exponential function]\label{exp shift}
Let $\gamma\in \Q$ with $\gamma\ge -1$. Assume $\beta\in\Q$ with $\beta\neq 0$. Then the real number $\exp_{\gamma}(1/\beta)$ is irrational and its irrationality exponent is equal to $2$.
\end{corollary}

\begin{remark}
  Note that the function $\exp_\gamma(z)$ is transcendental, so by Siegel-Shidlovskii's theorem 
  (see for example \cite[Theorem 5.5]{FelNest1998}), the number
  $\exp_\gamma(1/\beta)$ is a transcendental real number.
\end{remark}

\section{Explicit construction of Pad\'{e} approximants and Poincar\'{e}-type recurrence}\label{padeformal}

Throughout the article, the letter $n$ denotes a non-negative rational integer. Unless stated otherwise, the Landau symbols $\GrO$ and small $o$ refer when $n$ tends to infinity. Let $K$ be a field of characteristic $0$ and let $\alpha,\gamma,\delta\in K$ with $\gamma \notin \{-2,-3,\dots\}$. The Laurent series
\[
    f(z):=\sum_{k=0}^{\infty}\dfrac{\prod_{i=1}^k(\alpha i-\delta)}{(\gamma+2)_k}\dfrac{1}{z^{k+1}}
\]
satisfies $L(f)=0$, where $L$ is the differential operator
\begin{align*}
    L:=\dfrac{d}{dz}\left(-(z-\alpha)z\dfrac{d}{dz}+\gamma z+\delta \right)\in K\Big[z,\frac{d}{dz}\Big].
\end{align*}
The goal of this section is to construct explicit Pad\'e approximants $(P_{n,0},P_{n,1})_{n\geq 0}$ of $f(z)$ and to show that the sequences $(P_{n,0})_{n\geq 0}$ and $(P_{n,1})_{n\geq 0}$ satisfy
a certain Poincar\'e-type recurrence
of order $2$. This is done respectively in Proposition~\ref{pade f} and Lemma~\ref{rec rel}. For that purpose, define the $K$-homomorphism $\varphi_f:K[t]\rightarrow K$ by
\[
    \varphi_f(t^k) = \dfrac{\prod_{i=1}^k(\alpha i-\delta)}{(\gamma+2)_k} \qquad (k\geq 0).
\]
The above function extends naturally in a $K[z]$-homomorphism $\varphi_f: K[z,t]\rightarrow K[z]$, and then to a $K[z]$-homomorphism $\varphi_f: K[z,t][[1/z]]\rightarrow K[z][[1/z]]$. With this notation, the formal Laurent series $f(z)$ satisfies the following crucial identity
\[
    f(z)=\varphi_f \left(\dfrac{1}{z-t}\right).
\]
We denote by $(1/z^{\ell})$ the ideal of $K[[1/z]]$ generated by $1/z^{\ell}$ for $\ell\in\N$. We first establish some useful properties satisfied by $\varphi_f$.

\begin{lemma} \label{decompose}
    Let $n\in\Z, n\geq 0$. Define the differential operator $RD_n$ by
    \[
        RD_n=\dfrac{1}{n!}\left(\dfrac{d}{dz}+\dfrac{\gamma z+\delta}{(z-\alpha)z}\right)^n\kern-5.5pt(z-\alpha)^nz^n\in K(z)\Big[\frac{d}{dz}\Big].
    \]
    Then, in the ring $K(z)[\tfrac{d}{dz}]$, we have the identity~$:$
    \[
        RD_n=\dfrac{1}{n!}RD_{1}(RD_{1}+2z-\alpha)\cdots (RD_{1}+(n-1)(2z-\alpha)).
    \]
\end{lemma}

\begin{proof}
Set $a(z)=(z-\alpha)z$ and $b(z)=\gamma z+\delta$. Since we have
\begin{align*}
    \left(\dfrac{d}{dz}+\dfrac{b(z)}{a(z)}\right)a(z)^n &=\left[a(z)^{n-1}\left(\dfrac{d}{dz}+\dfrac{b(z)}{a(z)}\right)+(n-1)a'(z)a(z)^{n-2}\right]a(z)\\
    &=a(z)^{n-1}(RD_1+(n-1)a'(z)),
\end{align*}
where $a'(z)$ denotes the derivative of $a(z)$, we obtain the assertion.
\end{proof}

\begin{lemma} \label{kernel}
    Define the differential operator
    \begin{align*}
    &\mathcal{E}=\dfrac{d}{dt}+\dfrac{\gamma t+\delta}{t(t-\alpha)}\in K(t)\Big[\frac{d}{dt}\Big].
    \end{align*}
    We have $\mathcal{E}(I)\subseteq {\rm{ker}}\, \varphi_f$, where $I$ denotes the ideal of $K[t]$ generated by $t(t-\alpha)$.
    More generally, for each non-negative integers $k,n$ and each $P(t)\in I^n$, we have
    \begin{align}\label{eq: formula phi_f}
        \frac{1}{n!} \varphi_f\big(t^k\mathcal{E}^n(P(t))\big) = (-1)^n\binom{k}{n} \varphi_f\big(t^{k-n}P(t)\big).
    \end{align}
    In particular, if $k<n$, then $\varphi_f\big(t^k\mathcal{E}^n(P(t))\big) = 0$.
\end{lemma}

\begin{proof}
    Let $k,n$ be non-negative integers. We first prove that $\mathcal{E}(I)\subseteq {\rm{ker}}\, \varphi_f$. It is sufficient to prove that $\mathcal{E}(t^{k+1}(t-\alpha))\in {\rm{ker}}\, \varphi_f$. By definition of $\mathcal{E}$, we have $\mathcal{E}(t^{k+1}(t-\alpha))=(k+2+\gamma)t^{k+1}-(\alpha(k+1)-\delta)t^k$. Hence
    \begin{align*}
        \varphi_f(\mathcal{E}(t^{k+1}(t-\alpha)))&=(k+2+\gamma)\dfrac{\prod_{i=1}^{k+1}(\alpha i-\delta)}{(\gamma+2)_{k+1}}-(\alpha(k+1)-\delta)\dfrac{\prod_{i=1}^k(\alpha i-\delta)}{(\gamma+2)_k}=0.
    \end{align*}
    The identity \eqref{eq: formula phi_f} trivially holds if $n=0$. From above, it also holds if $k=0$ (both sides are then equal to $0$), so we may assume $k,n > 0$. Since $t^k\mathcal{E}=\mathcal{E}t^k-kt^{k-1}$ and $t^k\mathcal{E}^{n-1}(I^n)\subseteq I$, we find
    \begin{align*}
         \varphi_f\big(t^k\mathcal{E}^n(P(t))\big) =  \varphi_f\big(\mathcal{E}(t^k\mathcal{E}^{n-1}(P(t)))\big) - k \varphi_f\big(t^{k-1}\mathcal{E}^{n-1}(P(t))\big)
         = - k \varphi_f\big(t^{k-1}\mathcal{E}^{n-1}(P(t))\big).
    \end{align*}
    Then by induction, the identity \eqref{eq: formula phi_f} follows.
\end{proof}

\begin{lemma} \label{cal n-coeff R_n}
    For non-negative integers $n, m$, we have
    \begin{align} \label{eq:lem coeff R_n}
        \varphi_f(t^m(t-\alpha)^n)=\dfrac{(-1)^n\prod_{i=1}^m(\alpha i-\delta)\prod_{j=1}^n(\alpha(\gamma+j)+\delta)}{(\gamma+2)_{n+m}}.
    \end{align}
\end{lemma}

\begin{proof}
     Set $A_{m,n}=\varphi_f(t^m(t-\alpha)^n)$ and denote by $\widetilde{A}_{m,n}$ the right-hand side of \eqref{eq:lem coeff R_n}. We need to prove that
    \begin{align} \label{equality A}
        A_{m,n}=\widetilde{A}_{m,n}.
    \end{align}
    It is trivial if $n=0$ by definition of $\varphi_f$. Note that the sequence $(A_{m,n})_{n,m\geq 0}$ satisfies the following recurrence relation:
    \begin{align}
    A_{m,n+1}&=\varphi_f(t^{m}(t-\alpha)^{n+1})=\varphi_f(t^{m+1}(t-\alpha)^{n}-\alpha t^{m}(t-\alpha)^{n})=A_{m+1,n}-\alpha A_{m,n} \label{recurrence A}
    \end{align}
    for each $n,m\geq 0$. To prove $(\ref{equality A})$, it suffices to show that $(\widetilde{A}_{m,n})_{n,m\geq 0}$ satisfies the same recurrence, which follows by a straightforward calculation.
\end{proof}

Using the identity $z^k-t^k = (z-t)\sum_{\ell=0}^{k-1}t^{k-1-\ell}z^{\ell}$ for each $k\geq 1$, we deduce the following lemma.

\begin{lemma} \label{trivial}
    Let $P(z)=\sum_{k=0}^np_{k}z^k\in K[z]$. Then we have
    \[
        \dfrac{P(z)-P(t)}{z-t}=\sum_{\ell=0}^{n-1}\left(\sum_{k=\ell}^{n-1}p_{k+1}t^{k-\ell}\right)z^{\ell}.
    \]
\end{lemma}

\begin{proposition} \label{pade f}
    Let $n\geq 0$ be an integer.

    \smallskip

    $({\rm{i}})$ Define the polynomials $P_{n,0}(z)$ and $P_{n,1}(z)$ by
    \begin{align*}
    &P_{n,0}(z)=\dfrac{1}{n!}\sum_{k=0}^n(-1)^k(n+\gamma+1)_{n-k}\binom{n}{k}\left(\prod_{i=0}^{k-1}(\alpha(n-i)-\delta)\right)z^{n-k},\\
    &P_{n,1}(z)=\sum_{\ell=0}^{n-1}\left(\sum_{k=\ell}^{n-1}(-1)^{n-k-1}\dfrac{(n+\gamma+1)_{k+1}}{(k+1)!}\dfrac{\prod_{i=0}^{n-k-2}(\alpha(n-i)-\delta)}{(n-k-1)!}\dfrac{\prod_{j=1}^{k-\ell}(\alpha j-\delta)}{(\gamma+2)_{k-\ell}}\right)z^{\ell}.
    \end{align*}
    Then $(P_{n,0}(z),P_{n,1}(z))$ is a Pad\'{e} approximant of $f(z)$ of weight $n$.

    \medskip

    $({\rm{ii}})$ Denote by $R_n(z)= P_{n,0}(z)f(z)-P_{n,1}(z)$. We have the formula
    \[
        R_n(z)=\sum_{k=n}^{\infty}\binom{k}{n}\dfrac{\prod_{i=1}^k(\alpha i-\delta)\prod_{j=1}^n(\alpha(\gamma+j)+\delta)}{(\gamma+2)_{n+k}}\dfrac{1}{z^{k+1}}.
    \]
\end{proposition}

\begin{proof}
    $({\rm{i}})$ Recall that
    \[
        RD_n=\dfrac{1}{n!}\left(\dfrac{d}{dz}+\dfrac{\gamma z+\delta}{z(z-\alpha)}\right)^nz^n(z-\alpha)^{n}\in K\Big[z,\frac{d}{dz}\Big]
    \]
    and consider the polynomials
    \[
        P_{n,0}(z):= RD_n(1)\AND P_{n,1}(z):=\varphi_f\left(\dfrac{P_{n,0}(z)-P_{n,0}(t)}{z-t}\right).
    \]
    We first prove that $(P_{n,0},P_{n,1})$ is a Pad\'{e} approximant of $f(z)$ of weight $n$. We have
    \[
        P_{n,0}(z)f(z)-P_{n,1}(z)=\varphi_f\left(\dfrac{P_{n,0}(t)}{z-t}\right)=\sum_{k=0}^{\infty}\dfrac{\varphi_f(t^kP_{n,0}(t))}{z^{k+1}},
    \]
    it is therefore sufficient to show that
    \begin{align*}\label{relations}
        \varphi_f(t^kP_{n,0}(t))=0 \ \  \text{for} \ \ 0\le k \le n-1.
    \end{align*}
    Since $n!P_{n,0}(t) = \mathcal{E}^n(t^n(t-\alpha)^n)$, it is a direct consequence of Lemma \ref{kernel}. We now prove that
    \[
        RD_n(1)=\dfrac{1}{n!}\sum_{k=0}^n(-1)^k(n+\gamma+1)_{n-k}\binom{n}{k}\left(\prod_{i=0}^{k-1}(\alpha(n-i)-\delta)\right)z^{n-k}.
    \]
    Let $n,m$ be non-negative integers and set
    \[
        P^{(m)}_{n,0}(z):=RD_n(z^m),
    \]
    so that $P_{n,0}(z)=P_{n,0}^{(0)}(z)$. We also define
    \[
        \widetilde{P}^{(m)}_{n,0}(z):=\dfrac{1}{n!}\sum_{k=0}^{n}(-1)^{k}(n+m+1+\gamma)_{n-k}\binom{n}{k}  \left(\prod_{i=0}^{k-1}(\alpha(n+m-i)-\delta)\right)z^{n+m-k},
    \]
    with the convention that $\prod_{i=0}^{k-1}(\alpha(n+m-i)-\delta)=1$ if $k=0$. We claim that
    \begin{align} \label{equality}
        P^{(m)}_{n,0}(z)=\widetilde{P}^{(m)}_{n,0}(z).
    \end{align}
    It is obvious when $n=0$ (the both sides are equal to $z^m$). For each non-negative integer $k$, we have
    \[
        (RD_1+k(2z-\alpha))(z^m) = (2(k+1)+m+\gamma)z^{m+1}-(\alpha(m+1+k)-\delta)z^m,
    \]
    and combined with Lemma $\ref{decompose}$, it yields
    \begin{align}
    P^{(m)}_{n,0}(z)&=\dfrac{1}{n!}RD_1(RD_1+2z-\alpha)\cdots (RD_1+(n-1)(2z-\alpha))(z^m) \nonumber\\
    &=\dfrac{1}{n!}RD_1(RD_1+2z-\alpha)\cdots (RD_1+(n-2)(2z-\alpha))[(2n+m+\gamma)z^{m+1}-(\alpha(m+n)-\delta)z^m] \nonumber\\
    &=\dfrac{1}{n}RD_{n-1}[(2n+m+\gamma)z^{m+1}-(\alpha(m+n)-\delta)z^m]\nonumber\\
    &=\dfrac{2n+m+\gamma}{n}P^{(m+1)}_{n-1,0}(z)-\dfrac{\alpha(m+n)-\delta}{n}P^{(m)}_{n-1,0}(z). \label{recurrence P}
    \end{align}
    To get $(\ref{equality})$, it remains to show that $(\widetilde{P}^{(m)}_{n,0}(z))_{n,m\geq 0}$ satisfies the recurrence relation $(\ref{recurrence P})$, which follows by a straightforward calculation. The identity for $P_{n,1}(z)$ is obtained by combining the formula obtained for $P_{n,0}(z)$ and Lemma~\ref{trivial}.

    \medskip

    $({\rm{ii}})$ By definition of $R_n(z)$ and by combining \eqref{eq: formula phi_f} with Lemma $\ref{cal n-coeff R_n}$, we find
    \begin{align*}
        R_n(z)&=\sum_{k=n}^{\infty}\dfrac{\varphi_f(t^kP_{n,0}(t))}{z^{k+1}} = (-1)^n\sum_{k=n}^{\infty}\binom{k}{n}\dfrac{\varphi_f(t^{k}(t-\alpha)^n)}{z^{k+1}}
        =\sum_{k=n}^{\infty}\binom{k}{n}\dfrac{\prod_{i=1}^k(\alpha i-\delta)\prod_{j=1}^n(\alpha(\gamma+j)+\delta)}{(\gamma+2)_{n+k}}\dfrac{1}{z^{k+1}}.
    \end{align*}
\end{proof}

The following corollary will be important to estimate the denominator of $P_{n,0}(\beta)$ and $P_{n,1}(\beta)$ when $\beta\in\Q$ and $f(z)$ is a binomial function as in Corollary~\ref{binom}.

\begin{corollary} \label{binomial pade}
    Let $\omega\in K\setminus \Z$.
    We use the same notation as in Proposition $\ref{pade f}$. Suppose $(\alpha,\gamma,\delta)=(1,-1,1+\omega)$. Then we have
    \begin{align*}
        &P_{n,0}(z)=\sum_{k=0}^n(-1)^{n-k}\binom{n+k-1}{k}\binom{n-\omega-1}{n-k}z^k,\\
        &P_{n,1}(z)=\sum_{k=0}^{n-1}(-1)^{n-1-k}\binom{n+k}{k}\binom{n+\omega}{n-1-k}z^k.
    \end{align*}
\end{corollary}

\begin{proof}
    The identity for $P_{n,0}(z)$ is directly obtained by Proposition \ref{pade f} by our choice of parameters. We now prove the identity for $P_{n,1}(z)$. For $\omega\in K\setminus \Z$ and non-negative integer $n$, we write $f_{\omega}(z)=1/z\cdot (1-1/z)^{\omega}$ and
    \[
        RD_{n,\omega}:=\dfrac{1}{n!}\left(\dfrac{d}{dz}+\dfrac{-z+1+\omega}{(z-1)z}\right)^n(z-1)^nz^n.
    \]
    Set $\widetilde{P}_{n,1}(z):=\displaystyle\sum_{k=0}^{n-1}(-1)^{k+1}\binom{n+k}{k}\binom{n+\omega}{n-1-k}z^{k}$. Note that, by a straightforward calculation, the coefficients of $z^{n-1}$, $z^{n-2}$, $z^{n-3}$ of $P_{n,1}(z)$ and $\widetilde{P}_{n,1}(z)$ are respectively the same. By \eqref{equality}, we get
    \begin{align} \label{equal tilde Q}
        \widetilde{P}_{n,1}(z)\cdot z^2=RD_{n-1,-\omega}(z^2).
    \end{align}
    Recall that for each $\ell\in\N$, we denote by $(1/z^{\ell})$ the ideal of $K[[1/z]]$ generated by $1/z^{\ell}$. Since $R_n(z)=P_{n,0}(z)f_{\omega}(z)-P_{n,1}(z) \in (1/z^{n+1})$, by multiplying $R_n(z)$ by $f_{\omega}(z)^{-1}=z^2f_{-\omega}(z)$, we obtain
    \begin{align}\label{reminder 1}
        P_{n,0}(z)-P_{n,1}(z)z^2f_{-\omega}(z)\in(1/z^{n}).
    \end{align}
    By $(\ref{equal tilde Q})$, the same argument in the proof of Proposition $\ref{pade f}$ ensures that there exists a polynomial $\widetilde{P}_{n,0}(z)$ with
    \begin{align}\label{reminder *}
        \widetilde{P}_{n,0}(z)-\widetilde{P}_{n,1}(z)z^2f_{\omega}(z) \in (1/z^{n}).
    \end{align}
    Subtracting (\ref{reminder *}) from (\ref{reminder 1}), we get
    \[
        (\widetilde{P}_{n,0}(z)-P_{n,0}(z))-(\widetilde{P}_{n,1}(z)z^2-P_{n,1}(z)z^2)f_{-\omega}(z)\in (1/z^{n}).
    \]
    Since the coefficients of $z^{n-1},z^{n-2},z^{n-3}$ of $P_{n,1}(z)$ and $\widetilde{P}_{n,1}(z)$ coincide, we find
    \[
        {\rm{deg}}\, \big(\widetilde{P}_{n,1}(z)z^2-P_{n,1}(z)z^2\big)\le n-2.
    \]

    It is known that all weight $n-1$ Pad\'{e} approximants $(P_0(z),P_1(z))$ of $f_{-\omega}(z)$ satisfy ${\rm{deg}}\,P_0=n-1$ (confer \cite[p.\,208, II, Theorem $1.\,2.\,2$]{J} and \cite[p.\,96]{MPerf}). Hence $\widetilde{P}_{n,1}(z)z^2-P_{n,1}(z)z^2=0$, which completes the proof.
\end{proof}

\begin{lemma} \label{lem: Delta_n}
    Let $P_{n,0}(z),P_{n,1}(z)$ be the polynomials defined in Proposition~\ref{pade f} and set
    \[
        M_{2,n}=\begin{pmatrix}P_{n,0}(z) & P_{n,1}(z) \\ P_{n+1,0}(z) & P_{n+1,1}(z)\end{pmatrix}.
    \]
    Then we have
    \[
        {\rm{det}} \,M_{2,n}=\dfrac{(n+\gamma+2)_{n+1}}{(n+1)!}\cdot \dfrac{\prod_{i=1}^n(\alpha i-\delta)\prod_{j=1}^n(\alpha(\gamma+j)+\delta)}{(\gamma+2)_{2n}}.
    \]
    In particular, under the conditions $\delta\notin \alpha \N$ and $-(\alpha\gamma+\delta)\notin \alpha \N$, we have ${\rm{det}} \,M_{2,n}\neq 0$.
\end{lemma}

\begin{proof}
    We have
    \begin{align*}
        -{\rm{det}} \,M_{2,n}=
        \left|\begin{array}{cc}
            P_{n,0}(z) & R_n(z)\\
            P_{n+1,0}(z) & R_{n+1}(z)
        \end{array}\right|
        =P_{n,0}(z)R_{n+1}(z)-P_{n+1,0}(z)R_n(z).
    \end{align*}
    Since ${\rm{deg}}\,P_{n,0}=n$ and $R_n(z)\in (1/z^{n+1})$ for each $n\in \N$, we have $P_{n,0}(z)R_{n+1}(z)\in 1/zK[[1/z]]$ and
    \begin{align*}
        &P_{n+1,0}(z)R_n(z)\in \dfrac{(n+\gamma+2)_{n+1}}{(n+1)!}\dfrac{\prod_{i=1}^n(\alpha i-\delta)\prod_{j=1}^n(\alpha(\gamma+j)+\delta)}{(\gamma+2)_{2n}}+1/zK[[1/z]].
    \end{align*}
    We conclude by noting that ${\rm{det}} \,M_{2,n}\in K[z]$.
\end{proof}

We now prove that the sequences $(P_{n,0}(z))_{n\geq 0}$ and $(P_{n,1}(z))_{n\geq 0}$ forming the Pad\'e approximants satisfy a certain recurrence relation. This will allow us to use the Poincar\'e-Perron theorem to estimate their growth, when evaluated at $z=\beta\in\Q$ with $\beta$ large enough.

\begin{lemma} \label{rec rel}
    Assume $\delta\notin \alpha \N$ and $-(\alpha\gamma+\delta)\notin \alpha \N$. Then the sequences $(P_{n,0}(z))_{n\ge 0}$, $(P_{n,1}(z))_{n\ge 0}$ and $(R_n(z))_{n\ge 0}$ defined as in Proposition~\ref{pade f} satisfy the recurrence
    \begin{align}\label{recurrence}
        A_nX_{n+1}-(z-B_n)X_n+C_nX_{n-1}=0 \quad (n\ge 1),
    \end{align}
    where for each $n\in \N$, we have
    \[
        A_n:=\dfrac{(n+\gamma+1)(n+1)}{(2n+\gamma+1)(2n+\gamma+2)} , \ \ B_n:=\dfrac{2\alpha n^2+2\alpha n(1+\gamma)+\gamma(\alpha-\delta)}{(2n+\gamma)(2n+2+\gamma)}, \ \
        C_n:=\dfrac{(\alpha n-\delta)(\alpha(\gamma+n)+\delta)}{(2n+\gamma)(2n+\gamma+1)}.
    \]
\end{lemma}

\begin{proof}
    Since ${\rm{deg}}\,P_{n,0}=n$, the sequence $(P_{n,0}(z))_{n\ge0}$ forms a $K$-basis of $K[z]$. Consequently, there exist $A_n, B_n, C_n,D_{i,n}\in \Q$ such that
    \[
        zP_{n,0}(z)=A_nP_{n+1,0}(z)+B_nP_{n,0}(z)+C_nP_{n-1,0}(z)+{\displaystyle{\sum_{i=0}^{n-2}}}D_{i,n}P_{i,0}(z).
    \]
    Write $P_{n,0}(z)=\sum_{k=0}^np_{n,k}z^k$ with $p_{n,k}\in K$. Then we get
    \begin{align*}
        A_n=\dfrac{p_{n,n}}{p_{n+1,n+1}}=\dfrac{(n+\gamma+1)(n+1)}{(2n+\gamma+1)(2n+\gamma+2)} \AND B_n=\dfrac{p_{n,n-1}}{p_{n,n}}-\dfrac{p_{n+1,n}}{p_{n+1,n+1}}=\dfrac{2\alpha n^2+2\alpha n(1+\gamma)+\gamma(\alpha-\delta)}{(2n+\gamma)(2n+2+\gamma)}.
    \end{align*}
    Note that the polynomials $(P_{n,0}(t))_{n\ge 0}$ form an orthogonal system with respect to the bilinear form
    \[
        \langle \, , \, \rangle_{f}:K[t]\times K[t]\longrightarrow K; \ \ (A(t),B(t))\mapsto \varphi_f(A(t)B(t)).
    \]
    Let us show that $D_{i,n}=0$ for $0\le i \le n-2$. Using the bilinear form $\langle\, , \, \rangle_{f}$, we find
    \[
        D_{i,n}=\dfrac{\langle tP_{n,0}(t),P_{i,0}(t)\rangle_{f}}{\langle P_{i,0}(t),P_{i,0}(t)\rangle_{f}}=\dfrac{\langle P_{n,0}(t), tP_{i,0}(t)\rangle_{f}}{\langle P_{i,0}(t),P_{i,0}(t)\rangle_{f}}=0 \ \ \ (0\le i \le n-2).
    \]
    Finally, we obtain
    \begin{align*}
        C_n =\dfrac{\langle tP_{n,0}(t), P_{n-1,0}(t)\rangle_{f}}{\langle P_{n-1,0}(t), P_{n-1,0}(t)\rangle_{f}}=\dfrac{\langle P_{n,0}(t), tP_{n-1,0}(t)\rangle_{f}}{\langle P_{n-1,0}(t), P_{n-1,0}(t)\rangle_{f}}
        =\dfrac{p_{n-1,n-1}}{p_{n,n}}\dfrac{\langle P_{n,0}(t), P_{n,0}(t)\rangle_{f}}{\langle P_{n-1,0}(t), P_{n-1,0}(t)\rangle_{f}}.
    \end{align*}
    Since $\langle P_{m,0}(t), P_{m,0}(t)\rangle_{f} = p_{m,m}\langle t^m, P_{m,0}(t)\rangle_{f} = p_{m,m}\varphi_{f}(t^mP_{m,0}(t))$ for any non-negative integer $m$, we find
        \begin{align*}
        C_n  =\dfrac{\varphi_{f}(t^nP_{n,0}(t)) }{\varphi_{f}(t^{n-1}P_{n-1,0}(t))}
        =\dfrac{(-1)^n\varphi_{f}(t^n(t-\alpha)^n)}{(-1)^{n-1}\varphi_{f}(t^{n-1}(t-\alpha)^{n-1})}
        =\dfrac{(\alpha n-\delta)(\alpha(\gamma+n)+\delta)}{(2n+\gamma)(2n+\gamma+1)},
    \end{align*}
    where the last equality comes from Lemma $\ref{cal n-coeff R_n}$. Consequently, $(P_{n,0}(z))_{n\ge 0}$ satisfies \eqref{recurrence}. We claim that $(P_{n,1}(z))_{n\ge 0}$ also satisfies \eqref{recurrence}. Indeed, since $P_{n,1}(z)= \varphi_f\left(\displaystyle\frac{P_{n,0}(z)-P_{n,0}(t)}{z-t}\right)$, we get
    \begin{align*}
        &A_nP_{n+1,1}(z)-(z-B_n)P_{n,1}(z)+C_nP_{n-1,1}(z)\\
        &=\varphi_f\left(\dfrac{A_n(P_{n+1,0}(z)-P_{n+1,0}(t))-(z-B_n)(P_{n,0}(z)-P_{n,0}(t))+C_n(P_{n-1,0}(z)-P_{n-1,0}(t))}{z-t}\right)\\
        &=\varphi_f\left(\dfrac{-A_nP_{n+1,0}(t)+(z-B_n)P_{n,0}(t)-C_nP_{n-1,0}(t)}{z-t}\right)\\
        &=\varphi_f(P_{n,0}(t))=0
    \end{align*}
    for each $n\ge 1$, and our claim follows. Finally, since  $(R_n(z))_{n\ge 0}$ is a linear combination of  $(P_{n,0}(z))_{n\ge 0}$ and $(P_{n,1}(z))_{n\ge 0}$, it obviously satisfies \eqref{recurrence}.
\end{proof}

\section{Denominator of quotients of Pochhammer symbols and related estimates}
\label{quotient}

The first objective of this section is to control the size of the denominator of $\{P_{n,0}(\beta),P_{n,1}(\beta)\}$, where $P_{n,0}(z)$, $P_{n,1}(z)$ are the polynomials of Proposition~\ref{pade f} and $\beta$ is a non-zero rational number (see Lemma~\ref{denominator P Q} below). For that purpose, we first establish an important estimate for the denominator of the \emph{quotient} of Pochhammer symbols, that modifies and generalizes \cite[Lemma $10$]{Lepetit}. Our second goal is to get precise estimates for the growth of the sequences $(P_{n,0}(\beta))_{n\geq 0}$, $(P_{n,1}(\beta))_{n\geq 0}$ and $(R_n(\beta))_{n\geq 0}$ (when defined), see Lemma~\ref{apply P-P} below. The proof of the latter result is based on the Poincar\'e-Perron theorem.

\begin{lemma}\label{valuation}
    Let $n\in\N$ and $\alpha,\beta\in \Q\setminus\{0\}$ such that $\alpha,\beta$ are not negative integers.

    $({\rm{i}})$ For $k = 0,\dots,n$, we have
    \[
        \nu_n(\alpha)\cdot \dfrac{(\alpha)_k}{k!}\in \Z \qquad\textrm{where }\nu_n(\alpha)=\prod_{\substack{q:\text{prime} \\ q|{\rm{den}}(\alpha)}} q^{n+\lfloor n/(q-1)\rfloor}.
    \]

    $({\rm{ii}})$ Denoting by $\varphi$ is the Euler's totient function, we have
    \[
        \limsup_{n\to \infty}\dfrac{1}{n}\log\,{\rm{den}}\left(\dfrac{0!}{(\beta)_0},\ldots,\dfrac{n!}{(\beta)_n}\right)\le \dfrac{{\rm{den}}(\beta)}{\varphi({\rm{den}}(\beta))}.
    \]

    $({\rm{iii}})$ We have
    \[
        \limsup_{n\to \infty}\dfrac{1}{n}\log\,{\rm{den}}\left(\dfrac{1}{\beta},\ldots,\dfrac{1}{\beta+n-1}\right)\le \dfrac{{\rm{den}}(\beta)}{\varphi({\rm{den}}(\beta))}.
    \]
\end{lemma}

\begin{proof}
    The property $({\rm{i}})$ is proven in \cite[Lemma $2.2$]{B}. We now prove $({\rm{ii}})$. The proof dates back to the book by C.~L.~Siegel \cite[p.81]{Siegel}. Define $d={\rm{den}}(\beta)$, $c=d\cdot \beta$ and
    \[
        D_n(\beta)={\rm{den}}\left(\dfrac{0!}{(\beta)_0},\ldots,\dfrac{n!}{(\beta)_n}\right).
    \]
    Given a non-negative integer $k$, we set $N_k=c(c+d)\cdots(c+(k-1)d)$.
    Let $p$ be a prime number with $p\mid N_k$. The following three properties hold.

    \medskip

    $({\rm{a}})$  We have ${\rm{GCD}}(p,d)=1$. For any integers $i, \ell$ with $\ell>0$, there exists exactly one integer $\nu$ with $0\le \nu \le p^{\ell}-1$ and such that $p^{\ell}\mid c+(i+\nu)d$.

    \medskip

    $({\rm{b}})$ Let $\ell$ be a strictly positive integer with $|c|+(k-1)d<p^{\ell}$. Then $N_k$ is not divisible by $p^{\ell}$.

    \medskip

    $({\rm{c}})$ Set $C_{p,k}=\lfloor \log(|c|+(k-1)d)/\log(p)\rfloor $. Then we have
    \[
        v_p(k!)=\sum_{\ell=1}^{C_{p,k}}\left \lfloor \dfrac{k}{p^{\ell}}  \right\rfloor\le v_p(N_k)\le \sum_{\ell=1}^{C_{p,k}}\left(1+\left \lfloor \dfrac{k}{p^{\ell}} \right\rfloor  \right)=v_p(k!)+C_{p,k},
    \]
    where $v_p$ denotes the $p$-adic valuation. We deduce that
    \[
        v_p\left(\dfrac{k!}{(\beta)_k}\right)=v_p\left(\dfrac{d^kk!}{N_k}\right)\ge
        \begin{cases}
        -C_{p,k} & \ \ \text{if} \ \ p \mid N_k\\
         0 & \ \ \text{otherwise},
        \end{cases}
    \]
    hence
    \[
        \log\,\left|\dfrac{k!}{(\beta)_k}\right|_p\le
        \begin{cases}
         C_{p,k}\log(p) & \ \ \text{if} \ \ p \mid N_k\\
         0 & \ \ \text{otherwise}.
        \end{cases}
    \]
    By the above identity,  we obtain
    \begin{align*}
        \log\,D_{n}=\sum_{p:\text{prime}}\max_{0\le k \le n} \log\,\left|\dfrac{k!}{(\beta)_k}\right|_p&\le \sum_{p\mid N_n}C_{p,n}\log(p)\le \log(|c|+(n-1)d)\sum_{p\mid N_n}1\\
        &=\log(|c|+(n-1)d)\pi_{|c|,d}(|c|+(n-1)d),
    \end{align*}
    where $\pi_{|c|,d}(x)=\#\{p:\text{prime}~;~ p\equiv |c| \ \text{mod} \ d, \ p<x\}$ for $x>0$. By Dirichlet's prime number theorem for arithmetic progressions, we have
    \[
        \limsup_{n\to \infty}  \dfrac{\log(|c|+(n-1)d)\pi_{|c|,d}(|c|+(n-1)d)}{n}=\dfrac{d}{\varphi(d)},
    \]
    and we deduce $({\rm{ii}})$. The property $({\rm{iii}})$ follows by a similar argument.
\end{proof}

\begin{lemma} \label{denominator P Q}
    Let $\alpha,\beta,\gamma,\delta\in\Q$ with $\beta\neq 0$ and $\gamma \geq -1$, $x\in \Q\cap [0,1)$ and $\omega\in \Q\setminus \Z$.
    We denote by $(P_{n,0}(z))_{n\geq 0}$ and $(P_{n,1}(z))_{n\geq 0}$ the sequences of polynomials defined as in Proposition~\ref{pade f}. Fix $n\in \N$ and set
    \[
        \nu_n(\gamma)=\prod_{\substack{q:\text{prime} \\ q\mid{\rm{den}}(\gamma)}} q^{n+\lfloor n/(q-1)\rfloor}, \ \ D_n(\gamma)={\rm{den}}\left(\dfrac{0!}{(\gamma+2)_0}, \ldots, \dfrac{(n-1)!}{(\gamma+2)_{n-1}}\right), \ \
        {d_n(x)={\rm{den}}\left(\dfrac{1}{1+x},\ldots,\dfrac{1}{n+x}\right)}.
    \]

    $({\rm{i}})$ Suppose $\alpha\neq 0$. Then $\kappa_n P_{n,0}(\beta)$, $\kappa_n P_{n,1}(\beta)\in\Z$, where
    \[
        \kappa_n:= \nu_n(\gamma)\nu_{n}(\delta/\alpha)D_n(\gamma){\rm{den}}(\alpha)^{n}{\rm{den}}(\beta)^n.
    \]
    $({\rm{ii}})$ Suppose $\alpha = 1$, and that $(\gamma,\delta)$ is equal to $(x,-x)$ or $(-1,1+\omega)$. Then $\kappa_n P_{n,0}(\beta)$, $\kappa_n P_{n,1}(\beta)\in\Z$, where
    \[
        \kappa_n:=
        \left\{\begin{array}{ll}
            {\rm{den}}(x)\nu_{n}(x)d_n(x){\rm{den}(\beta)^n} & \textrm{if $\gamma=x,\delta=-x$ $($shifted logarithmic case$)$} \\
            ~ & ~ \\
            \dfrac{\nu_{n}(\omega){\rm{den}}(\beta)^n}{G_n(\omega)} & \textrm{if $\gamma=-1,\delta=1+\omega$ $($binomial case$)$}.
        \end{array}\right.
    \]
    Here, $G_n(\omega)$ denotes the positive integer defined as in Corollary $\ref{binom}$.

    \medskip

    $({\rm{iii}})$ Suppose $\alpha=0$.  Then $\kappa_n P_{n,0}(\beta)$, $\kappa_n P_{n,1}(\beta)\in\Z$, where
    \[
        \kappa_n:= \nu_n(\gamma)D_n(\gamma){\rm{den}}(\delta)^n{\rm{den}}(\beta)^nn!.
    \]
\end{lemma}

\begin{proof}
    $({\rm{i}})$ By Proposition $\ref{pade f}$, we have
    \begin{align*}
    &P_{n,0}(z)=\sum_{k=0}^n(-\alpha)^k\dfrac{(n+\gamma+1)_{n-k}}{(n-k)!}\dfrac{(n-k+1-\delta/\alpha)_{k}}{k!}z^{n-k},\\
    &P_{n,1}(z)=\sum_{\ell=0}^{n-1}\alpha^{n-1-\ell}\left(\sum_{k=\ell}^{n-1}(-1)^{n-k-1}\dfrac{(n+\gamma+1)_{k+1}}{(k+1)!}\dfrac{(k+2-\delta/\alpha)_{n-k-1}}{(n-k-1)!}\dfrac{(1-\delta/\alpha)_{k-\ell}}{(k-\ell)!}\dfrac{(k-\ell)!}{(\gamma+2)_{k-\ell}}\right)z^{\ell}.
    \end{align*}
    Applying Lemma $\ref{valuation}$ $({\rm{i}})$, we find
    \begin{align*}
    &\nu_n(\gamma)\nu_{n}(\delta/\alpha)D_n(\gamma){\rm{den}}(\alpha)^{n} \cdot (-\alpha)^k\dfrac{(n+\gamma+1)_{n-k}}{(n-k)!}\dfrac{(n-k+1-\delta/\alpha)_{k}}{k!}\in \Z
    \end{align*}
    for $0\le k \le n$, as well as
    \begin{align*}
    &\nu_n(\gamma)\nu_{n}(\delta/\alpha)D_n(\gamma){\rm{den}}(\alpha)^{n} \cdot \alpha^{n-1-\ell}\dfrac{(n+\gamma+1)_{k+1}}{(k+1)!}\dfrac{(k+2-\delta/\alpha)_{n-k-1}}{(n-k-1)!}\dfrac{(1-\delta/\alpha)_{k-\ell}}{(k-\ell)!}\dfrac{(k-\ell)!}{(\gamma+2)_{k-\ell}}\in \Z,
    \end{align*}
    for $0\le \ell \le n-1$ and $\ell\le k \le n-1$. Hence the assertion.

    \medskip

    $({\rm{ii}})$ In the case of $\alpha=1,\gamma=x,\delta=-x$ (shifted logarithmic case), we simply have
    \begin{align*}
    &P_{n,0}(z)=\sum_{k=0}^n(-1)^k\dfrac{(n+1+x)_{n-k}}{(n-k)!}\dfrac{(n-k+1+x)_{k}}{k!}z^{n-k},\\
    &P_{n,1}(z)=\sum_{\ell=0}^{n-1}\left(\sum_{k=\ell}^{n-1}(-1)^{n-k-1}\dfrac{(n+1+x)_{k+1}}{(k+1)!}\dfrac{(k+2+x)_{n-k-1}}{(n-k-1)!}\dfrac{(1+x)}{k-\ell+1+x}\right)z^{\ell}.
    \end{align*}
    Applying once again Lemma $\ref{valuation}$ $({\rm{i}})$, we find the expected result. In the case of $\alpha=1,\gamma=-1,\delta=1+\omega$ (binomial case), the assertion follows from Corollary $\ref{binomial pade}$.

    \medskip

    $({\rm{iii}})$ Assume $\alpha=0$. By Proposition $\ref{pade f}$, we have
    \begin{align*}
    &P_{n,0}(z)=\sum_{k=0}^n\dfrac{(n+\gamma+1)_{n-k}}{(n-k)!}\dfrac{\delta^k}{k!}z^{n-k},\\
    &P_{n,1}(z)=\sum_{\ell=0}^{n-1}\left(\sum_{k=\ell}^{n-1}(-1)^{n-k-1}\dfrac{(n+\gamma+1)_{k+1}}{(k+1)!}
    \dfrac{(-\delta)^{n-\ell-1}}{(n-k-1)!}\dfrac{1}{(\gamma+2)_{k-\ell}}\right)z^{\ell}.
    \end{align*}
    We conclude once again by Lemma $\ref{valuation}$ $({\rm{i}})$.
\end{proof}

Recall that by Lemma $\ref{rec rel}$, the sequences $(P_{n,0}(z))_{n\geq 0}$, $(P_{n,1}(z))_{n\geq 0}$  and $(R_n(z))_{n\geq 0}$ of Proposition~\ref{pade f} satisfy a Poincar\'e-type recurrence. Using the Poincar\'e-Perron theorem, it allows us to describe precisely the asymptotic behavior of those sequences evaluated at a rational number $\beta$ with $|\beta|>|\alpha|$.

\begin{lemma} \label{apply P-P}
    Let $\alpha,\beta,\gamma,\delta\in \Q$ with $\gamma\ge -1$ and  $\delta\notin \alpha \N$ and $-(\alpha\gamma+\delta)\notin \alpha \N$.
    For a given $z\in\C$, let $\rho_1(z)\leq \rho_2(z)$ denote the moduli of the roots $2z-\alpha \pm 2\sqrt{z^2-z\alpha}$ of the polynomial equation
    \begin{equation} \label{eq: pol P}
        P(X) = X^2-2(2z-\alpha)X+\alpha^2.
    \end{equation}
    Fix $z\in\C$ with $|z|> |\alpha|$. Then we have $\rho_1(z)\leq |\alpha| < \rho_2(z)$, and we obtain, as $n$ tends to the infinity$:$
    \begin{align}
    \label{eq: estimates for P_n and Q_n}
      \max\{|P_{n,0}(z)|, |P_{n,1}(z)|\} \leq \rho_2(z)^{n(1+o(1))},
    \end{align}
    and
    \begin{equation} \label{eq: estimates for R_n}
      |R_n(z)| \leq
      \left\{
        \begin{array}{cc}
           \displaystyle \frac{n^r}{(2n)!} \left|\frac{\delta^2}{z}\right|^n &  \ \ \textrm{if $\alpha=0$} \\
           ~ & ~\\
           \medskip \rho_1(z)^{n(1+o(1))} & \ \ \textrm{if $\alpha\neq 0$,}
        \end{array}
      \right.
    \end{equation}
    where $r > 0$ is a constant which is independent of $n$.
\end{lemma}

\begin{proof}
    Fix  $n\geq 1$ and $z\in\C$ with $|z| > |\alpha|$. By Proposition~\ref{pade f}, we have
    \begin{align}
        \label{eq: proof estimate R_n: inter 0}
        R_n(z) = \sum_{k=n}^{\infty}   \frac{\lambda_{n,k}}{z^{k+1}}, \qquad \textrm{where } \lambda_{n,k}:=
        \binom{k}{n}\frac{\prod_{i=1}^{k}(\alpha i-\delta)\prod_{j=1}^{n}(\alpha(\gamma+j)+\delta)}{(\gamma+2)_{n+k}}.
    \end{align}
    Note that for each $k\geq n$, we have
    \begin{align*}
        |\lambda_{n,k+1}| = \frac{(k+1)|\alpha(k+1)-\delta|}{(k+1-n)|\gamma+n+k+2|} |\lambda_{n,k}| = (|\alpha|+o(1))|\lambda_{n,k}|
    \end{align*}
    as $k$ tends to the infinity. We deduce that $R_n(z)$ is absolutely convergent since $|z| > |\alpha|$. By Lemma~\ref{rec rel}, the sequences $(P_{n,0}(z))_{n\geq 0}$, $(P_{n,1}(z))_{n\geq 0}$ and $(R_n(z))_{n\geq 0}$ satisfy the recurrence relation \eqref{recurrence}, whose characteristic polynomial is precisely $P(X)$ defined in \eqref{eq: pol P}. As a consequence of Lemma~\ref{lem: Delta_n}, none of the three above sequences has two successive elements equal to $0$.

    \medskip

    First, we prove that the condition $|z|>|\alpha|$ implies that $\rho_1(z) \leq |\alpha| < \rho_2(z)$. If $\alpha=0$ it is trivial since in that case the roots of $P$ are $0$ and $4z$.
    Otherwise, suppose by contradiction that $\rho_1(z)=\rho_2(z)=|\alpha|$. Then
   \begin{align*}
    2|\alpha| = \rho_1(z)+\rho_2(z) \geq 2|2z-\alpha| \geq 2(2|z|-|\alpha|) > 2|\alpha|,
   \end{align*}
   a contradiction. Since $\rho_1(z) < \rho_2(z)$, we can apply the Poincar\'e-Perron theorem for the recurrence \eqref{recurrence}. For each solution $(X_n)_{n\geq 0}$ satisfying $X_n\neq 0$ for infinitely many $n$, there exists $i\in\{1,2\}$ such that $|X_n| = \rho_i(z)^{n(1+o(1))}$, hence \eqref{eq: estimates for P_n and Q_n}. It remains to estimate  $R_n(z)$.

   \medskip

    \textbf{Case I: $\alpha = 0$.} Then, for each integers $n,k$ with $k\geq n \geq |\gamma|$, we have
    \begin{align*}
        |\lambda_{n,k+1}| \leq \frac{|\delta|}{k+1-n}|\lambda_{n,k}|.
    \end{align*}
    In particular, given $n\geq |\gamma|$ and $j\geq 0$, we find $|\lambda_{n,n+j}| \leq |\delta|^j|\lambda_{n,n}|/j!$, and we deduce the estimate
    \begin{align}
        \label{eq inter 1 estimate R_n 1}
        |R_n(z)| \leq  \frac{|\lambda_{n,n}|}{|z|^{n+1}}e^{|\delta|/|z|}.
    \end{align}
    Let $C$ be the smallest integer $> |\gamma|$. We may assume $2n > C$. Then, there exists a constant $c_1 >0$, depending only on $\gamma$ and $C$, such that
    \begin{align*}
        |\lambda_{n,n}| = \frac{|\delta|^{2n}}{|(\gamma+2)_{2n}|} \leq c_1\frac{|\delta|^{2n}}{(2n-C)!} \leq c_1\frac{|\delta|^{2n}(2n)^C}{(2n)!}.
    \end{align*}
    Combining the above with \eqref{eq inter 1 estimate R_n 1}, we obtain the first estimate in \eqref{eq: estimates for R_n}.

    \medskip

    \textbf{Case II: $\alpha\neq 0$}. We define $C$ as the smallest integer $C > |\delta/\alpha|+|\gamma|$. Then, there exists a constant $c_1 >0$, depending only on $\gamma$ and $C$, such that
    \begin{align}
        \label{eq inter 2 estimate R_n 1}
         |\lambda_{n,k}| = \binom{k}{n}|\alpha|^{n+k} \frac{|(1-\delta/\alpha)_k(1+\gamma+\delta/\alpha)_n|}{|(\gamma+2)_{k+n}|}\leq c_1 |\alpha|^{n+k}\frac{k!(k+C)!(n+C)!}{(k-n)!(k+n-C)!n!}.
    \end{align}
    We may assume $n\geq 2C$, so that $k+n-C \geq k+C$. Then, we find
    \begin{align*}
        \frac{k!(k+C)!}{(k-n)!(k+n-C)!} = \frac{(k-n+1)\cdots k}{(k+C+1)\cdots(k+n-C)} &\leq (k-n+1)\cdots (k-n+2C) \\
        & \leq (k-n+2C)^{2C}.
    \end{align*}
    Together with \eqref{eq inter 2 estimate R_n 1} and $(n+C)!/n! \leq (n+C)^{2C}$, we obtain
    \begin{align*}
        \frac{|\lambda_{n,n+j}|}{|z|^{n+j}} \leq c_1(n+C)^{2C}\left|\frac{\alpha^2}{z}\right|^{n} (j+2C)^{2C} \left|\frac{\alpha}{z}\right|^j,
    \end{align*}
    for each $n,j$ with $n\geq 2C$ and $j\geq 0$. Combining the above with the identity \eqref{eq: proof estimate R_n: inter 0}, we deduce
    \begin{align*}
        |R_n(z)| \leq c_1(n+C)^{2C}\left|\frac{\alpha^2}{z}\right|^n \sum_{j\geq 0}(j+2C)^{2C} \left|\frac{\alpha}{z}\right|^j \leq c_2(n+C)^{2C}\left|\frac{\alpha^2}{z}\right|^n,
    \end{align*}
    where $c_2$ does not depend on $n$. In particular, we have
    \begin{align*}
        \limsup_{n\rightarrow\infty} \frac{\log |R_n(z)|}{n} \leq \left|\frac{\alpha^2}{z}\right| < |\alpha|.
    \end{align*}
    The Poincar\'e-Perron theorem together with $\rho_2(z) > |\alpha|$, ensures that we must have
    \[
        \lim_{n\rightarrow\infty} \frac{\log |R_n(z)|}{n} = \rho_1(z),
    \]
    hence the second estimate in \eqref{eq: estimates for R_n}.
\end{proof}

\section{Proof of the main theorem}\label{mainproof}

Corollary \ref{exp shift} is a direct consequence of our main theorem. We explain how to get Corollaries~\ref{log shift} and~\ref{binom} in the proof of Theorem $\ref{main}$ below.

\begin{proof} [Proof of Theorem $\ref{main}$]
Let $\alpha$, $\beta$, $\gamma$, $\delta$ be as in Theorem $\ref{main}$ and let $(P_{n,0}(z))_{n\geq 0}$, $(P_{n,1}(z))_{n\geq 0}$ and $(R_n(z))_{n\geq 0}$ be the sequences constructed in Proposition $\ref{pade f}$ for this choice of parameters. We first suppose $\alpha\neq 0$. Fix $n\in\N$ and define $p_n=\kappa_n P_{n,0}(\beta)$ and $q_n= \kappa_n P_{n,1}(\beta)$, where
\[
    \kappa_n:= \nu_n(\gamma)\nu_{n}(\delta/\alpha)D_n(\gamma){\rm{den}}(\alpha)^{n}{\rm{den}}(\beta)^n.
\]
By Lemma $\ref{lem: Delta_n}$ and $\ref{denominator P Q}$ respectively, the point $(p_n,q_n)$ is not proportional to $(p_{n+1},q_{n+1})$ and $(p_n,q_n)\in \Z^{2}$. According to Lemma $\ref{valuation}$, as $n$ tends to infinity, we have the estimate
\[
    \frac{\log\kappa_n}{n}=
   \log\nu(\gamma)\log\nu(\delta/\alpha)\frac{\rm{den}(\gamma)}{\varphi(\rm{den}(\gamma))}\rm{den}(\alpha){\rm{den}}(\beta)+o(1).
\]
Combined with Lemma $\ref{apply P-P}$, we finally get
\[
    |p_n|, |q_n|\le Q^{n+o(n)} \AND |p_nf(\beta)-q_n|\le E^{-n+o(n)},
\]
where the quantities $E$ and $Q$ are defined as in the statement of Theorem~\ref{main}. The conclusion of the theorem follows by applying Corollary \ref{cor lem: Alladi-Rob general}. To obtain the Corollaries \ref{log shift} and \ref{binom}, it suffices to replace the factor $\kappa_n$ by the one given by Lemma~\ref{denominator P Q} $\rm{(ii)}$.

\medskip

Suppose now $\alpha = 0$ and define $p_n=\kappa_n P_{n,0}(\beta)$ and $q_n= \kappa_n P_{n,1}(\beta)$, where
\[
    \kappa_n:= \nu_n(\gamma)D_n(\gamma){\rm{den}}(\delta)^n{\rm{den}}(\beta)^nn!.
\]
Once again $(p_n,q_n)$ is not proportional to $(p_{n+1},q_{n+1})$, and $(p_n,q_n)\in \Z^{2}$. According to Lemma $\ref{valuation}$, we have $\log(\kappa_n)/n = \log n + \GrO(1)$ as $n$ tends to infinity. Combined with Lemma $\ref{apply P-P}$, we finally get
\[
    |p_n|, |q_n|\le Q_n \AND |p_nf(\beta)-q_n|\le E_n^{-1},
\]
where, as $n$ tends to infinity, we have
\[
    \frac{\log Q_n}{n} = \log(n)+ \GrO(1)  \AND \frac{\log E_n}{n} = \log(n)+ \GrO(1).
\]
By applying Corollary \ref{cor lem: Alladi-Rob general} we get $\mu(f(\beta)) = 2$.
\end{proof}

\begin{remark}\label{remark: remark effective measures}
    We keep the notation of the proof of Theorem $\ref{main}$ and suppose that $\alpha\neq 0$. If there exist two explicit positive constants $a$, $b$ such that
    \begin{align}\label{eq: rmk on main thm}
        |p_n|, |q_n|\le aQ^n \AND |p_nf(\beta)-q_n|\le bE^{-n},
    \end{align}
    then the second part of Theorem~\ref{lem: Alladi-Rob general} yields an effective irrationality measure of $f(\beta)$ in Theorem~\ref{main}, i.e. two (effective) positive constants $c$ and $q_0$, depending only on $f(z)$, $\beta$ and the irrationality measure $\mu=1+\log Q/\log E$, such that
    \begin{equation*}
        \Big|f(\beta) - \frac{p}{q} \Big| \geq \frac{c}{q^\mu}
    \end{equation*}
    for each rational number $p/q$ with $q\geq q_0$. To get \eqref{eq: rmk on main thm}, it suffices to obtain effective upper bounds of the form $C\rho^n$ for the denominators $D_n(\gamma)$, as well as for the quantities $P_{n,0}(\beta)$, $P_{n,1}(\beta)$ and $R_n(\beta) = P_{n,0}(\beta)f(\beta)-P_{n,1}(\beta)$. For $D_n(\gamma)$, we can use an explicit version of Dirichlet's prime number theorem for arithmetic progressions (see \cite{Ram}) in the proof of Lemma~\ref{valuation}. The remaining estimates are obtained thanks to our refinement of the Poincar\'e-Perron theorem (see Theorem~\ref{thm: poincare-Perron effectif application} of Section~\ref{section: Poincare-Perron thm}).
\end{remark}

\section{Parametric geometry of numbers}\label{parametricgn}

In this section, the norm $\norm{\bx}$ of a vector or a matrix $\bx$ denotes the largest absolute value of its coefficients and $\bx\cdot\by$ stands for the standard scalar product between $\bx,\by\in\R^{s+1}$. We prove two generalizations in higher dimension of a well-known irrationality measure result for one real number (see for example \cite[Lemma 3]{A-R} and \cite{BennettAIP}). We also use parametric geometry of numbers to emphasize the duality between type I (linear forms) and type II (simultaneous) approximation.

\begin{definition}
  Let $s$ be a positive integer and let $\btheta=(\theta_0,\dots,\theta_s)$ be a point of $\R^{s+1}$ with $\theta_0=1$. The exponent of best simultaneous rational approximation $\lambda(\btheta)$ is the supremum of real numbers $\lambda\geq 0$ such that, for infinitely many $(q,p_1,\dots,p_s)\in\Z^{s+1}$ with $q> 0$, we have
  \begin{align*}
    \max_{1\leq i \leq s} |q\theta_i-p_i\theta_0| \leq q^{-\lambda}.
  \end{align*}
  The dual exponent $\omega(\btheta)$ is the supremum of real numbers $\omega\geq 0$ such that, for infinitely many non-zero $\by=(y_0,\dots,y_s)\in\Z^{s+1}$, we have
  \begin{align*}
    \max_{1\leq i \leq s} |y_0\theta_0+\cdots +y_s\theta_s| \leq \norm{\by}^{-\omega}.
  \end{align*}
\end{definition}

Note that by Dirichlet's Theorem, we have $\lambda(\btheta)\geq 1/s$ and $\omega(\btheta)\geq s$.

\begin{theorem}
    \label{lem: Alladi-Rob general}
    Let $s\geq 1$ be an integer and let $\btheta:=(\theta_0,\dots,\theta_s)\in\R^{s+1}$ with $\theta_0=1$. Suppose that there exist a sequence of matrices $(M_n)_{n\geq 0}$ in $\Mat_{s+1}(\Z)\cap \GL_{s+1}(\Q)$ and two unbounded increasing sequences
    $(Q_n)_{n\geq 0}$ and $(E_n)_{n\geq 0}$ of real numbers $\geq 1$ such that, for each $n$ and each row $\bx=(x_0,\dots,x_s)$ of $M_n$, we have
    \begin{align}
    \label{eq: lem Alladi-Rob general: eq0}
        |x_1|+\cdots+|x_s|\leq Q_n \AND |x_0\theta_0+\cdots+x_s\theta_s| \leq E_n^{-1}\,.
    \end{align}
    Then at least one of the numbers $\theta_1,\dots,\theta_s$ is irrational and
    \begin{align}
    \label{eq: lem Alladi-Rob general: eq1}
     \lambda(\btheta) \leq \limsup_{n\rightarrow\infty}\frac{\log Q_{n}}{\log E_{n-1}}\,.
    \end{align}
    Moreover, if there are positive constants $a,b,\alpha,\beta$ with $\alpha,\beta > 1$ such that $Q_n=a\alpha^n$ and $E_n= b^{-1}\beta^{n}$ for each $n\geq 0$, then for each integer point $(y_0,\dots,y_s)\in\Z^{s+1}$ with $2by_0 \geq 1$, we have
    \begin{align}
    \label{eq: lem Alladi-Rob general: eq2}
        \max_{1\leq i \leq s} |y_0\theta_i-y_i| \geq \frac{1}{cq^{\lambda}},
    \end{align}
    where $\lambda = \log \alpha /\log \beta$ and $c = 2a\alpha (2b)^\lambda$.
\end{theorem}

\begin{proof}
    Let $\by=(y_0,\dots,y_s)\in\Z^{s+1}$ with $2y_0 \geq E_0$ and let $n\geq 1$ be such that
    \begin{align*}
        E_{n-1} \leq 2y_0 \leq E_{n}.
    \end{align*}
    Since $\det M_n \neq 0$, there exists a row $\bx=(\bx_0,\dots,\bx_s)$ of $M_n$ such that $\bx\cdot \by$ is a non-zero integer. We deduce the lower bound
    \begin{align}\label{eq proof: lem Alladi-Robinson: eq1}
        \Big|\sum_{i=0}^{s}x_i(y_0\theta_i-y_i\theta_0)\Big| \geq |\bx\cdot\by| - y_0|\bx\cdot \btheta| \geq 1- y_0E_{n}^{-1} \geq \frac{1}{2}.
    \end{align}
    On the other hand, we have the upper bound
    \begin{align}\label{eq proof: lem Alladi-Robinson: eq2}
         \Big|\sum_{i=0}^{s}x_i(y_0\theta_i-y_i\theta_0)\Big| =  \Big|\sum_{i=1}^{s}x_i(y_0\theta_i-y_i\theta_0)\Big| \leq \sum_{i=1}^{s}|x_i|  \max_{1\leq j \leq s} |y_0\theta_j-y_j\theta_0| \leq Q_n \max_{1\leq i \leq s} |y_0\theta_i-y_i\theta_0|.
    \end{align}
    Together with \eqref{eq proof: lem Alladi-Robinson: eq1} and the inequality $Q_n = E_{n-1}^{\log Q_n/\log E_{n-1}} \leq (2y_0)^{\log Q_n/\log E_{n-1}}$, it leads us to
    \begin{align*}
         \frac{1}{2} \leq (2y_0)^{\log Q_n/\log E_{n-1}} \max_{1\leq i \leq s} |y_0\theta_i-y_i\theta_0|,
    \end{align*}
    hence \eqref{eq: lem Alladi-Rob general: eq1}. Suppose now that $Q_j=a\alpha^j$ and $E_j= b^{-1}\beta^{j}$ for each $j\geq 0$. Then $\beta^{n-1}\leq 2by_0$, and we write instead
    \[
        Q_n = a\alpha^n = a\alpha \beta^{(n-1)\log \alpha / \log \beta} \leq a\alpha(2by_0)^{\log \alpha / \log \beta}.
    \]
    Eq. \eqref{eq: lem Alladi-Rob general: eq2} follows by combining the above with \eqref{eq proof: lem Alladi-Robinson: eq1} and \eqref{eq proof: lem Alladi-Robinson: eq2}.
\end{proof}

We now give an alternative proof of the inequality \eqref{eq: lem Alladi-Rob general: eq1} of Theorem~\ref{lem: Alladi-Rob general}, which relies on parametric geometry of numbers (see \cite{S2009}, \cite{S2013} and \cite{R2015}). It is also possible to get an inequality of the type \eqref{eq: lem Alladi-Rob general: eq2} with an explicit constant $c>0$. The main interest of this approach is to show the duality between type I and type II approximation. Given a symmetric convex body $\cC$ of $\R^{s+1}$, we denote by $\lambda_1(\cC) \leq$ $\dots$ $\leq \lambda_{s+1}(\cC)$ the $(s+1)$ successive minima associated to $\cC$ with respect to the lattice $\Z^{s+1}$. The idea is as follows. The condition \eqref{eq: lem Alladi-Rob general: eq0} allows to bound from above the last minimum $\lambda_{s+1}$ of the family of symmetric convex bodies (related to type I approximation)
\begin{align*}
      \cC_{\btheta}(q):= \{\bx=(x_0,\dots,x_s)\in\R^{s+1}\,;\, |x_1|+\dots + |x_s| \leq  1 \AND |\bx\cdot \btheta| \leq e^{-q}\} \quad (q\geq 0).
\end{align*}
On the other hand, Mahler's duality implies that
\[
    1 \leq \lambda_1(\cC_{\btheta}^*(q))\lambda_{s+1}(\cC_{\btheta}(q)),
\]
where $\cC_{\btheta}^*(q)$ is the polar (or dual) convex body of $\cC_{\btheta}(q)$ defined as the set of points $\by=(y_0,\dots,y_s)\in\R^{s+1}$ such that $|\by\cdot\bx| \leq 1$ for each $\bx=(x_0,\dots,x_s)\in\cC_{\btheta}(q)$ (see for example \cite[Chap. VIII]{Cassels}). The key-point is to notice that $\cC_{\btheta}^*(q) \subseteq \cC_{\btheta}'(q) \subseteq 2\cC_{\btheta}^*(q)$, where $\cC_{\btheta}'(q)$ denotes the symmetric convex body (related to type II approximation)
\begin{align*}
  \cC_{\btheta}'(q):= \{\by=(y_0,\dots,y_s)\in\R^{s+1}\,;\, |y_0| \leq  e^q \AND \max_{1\leq i \leq s}|y_0\theta_i-y_i\theta_0| \leq 1\}\quad (q\geq 0).
\end{align*}
This can be proved by using the identity
\begin{align*}
  \bx\cdot\by = y_0(\bx\cdot\btheta) + \sum_{i=1}^{s}x_i(\theta_0y_i-y_0\theta_i).
\end{align*}
So, an upper bound for $\lambda_{s+1}(\cC_{\btheta}(q))$ leads to a lower bound for $\lambda_1(\cC_{\btheta}'(q))$ thanks to Mahler's duality. Lastly, the exponent $\lambda(\btheta)$ can be obtained by using its parametric analog through the classical identity
\begin{align}
\label{eq proof: lem alladi-Rob inter 1}
    \lambda(\btheta) = \frac{-\pu}{1+\pu}, \qquad \textrm{where } \pu:= \liminf_{q\rightarrow \infty} \frac{\log \lambda_1(\cC_{\btheta}'(q))}{q}
\end{align}
(it is a corollary of \cite[Theorem 1.4]{S2009}, also see \cite[Corollary 1.4]{R2015}). This will lead us to  \eqref{eq: lem Alladi-Rob general: eq1}.

\begin{proof}[Alternative proof of \eqref{eq: lem Alladi-Rob general: eq1} of Theorem~\ref{lem: Alladi-Rob general}]
    For simplicity we write $\lambda_{s+1}(q):=\lambda_{s+1}(\cC_{\btheta}(q))$ and  $\lambda_1^*(q):=\lambda_1(\cC_{\btheta}'(q))$. Given a non-zero $\bx\in\R^{s+1}$ and $q\geq 0$, we define $L_\bx(q)$ as the minimum of all $L\geq 0$ such that $\bx\in e^L\cC_{\btheta}(q)$. A short computation shows that
    \begin{align*}
        L_\bx(q):= \max\Big\{ \log \Big(\sum_{i=1}^{s}|x_i|\Big), \log |\bx\cdot \btheta| + q \Big\} \in \R,
    \end{align*}
    with the convention $\log(0) = -\infty$  (see \cite[Section 2.2]{R2015}). By definition of $\lambda_{s+1}$, we deduce that for any non-empty set $A\subset\Z^{s+1}$ containing at least $s+1$ linearly independent points, we have
    \begin{align}
    \label{eq proof: lem alladi-Rob inter 3}
        \log \lambda_{s+1}(q) \leq \max_{\bx\in A} L_\bx(q).
    \end{align}
    Fix an index $n\geq 0$. Our hypotheses on $M_n$ imply that the set $A_n \subset \Z^{s+1}$ of its rows contains $(s+1)$ linearly independent points. Moreover, for each $\bx\in A_n$ we have $L_\bx(q) \leq \max\{ \log Q_n, -\log E_n + q\}$ by \eqref{eq: lem Alladi-Rob general: eq0}. Using \eqref{eq proof: lem alladi-Rob inter 3} we get
    \begin{align*}
        \log \lambda_{s+1}(q) \leq P(q):= \min_{n\in\N}\big( \max\{ \log Q_n, -\log E_n + q\}\big) \quad (q\geq 0).
    \end{align*}
    The function $P$ is continuous, piecewise linear with slopes $0$ and $1$. More precisely, $P$ changes from slope $0$ to slope $1$ at $q_n' := \log(E_nQ_n)$ and from slope $1$ to slope $0$ at $q_n:=\log (E_{n}Q_{n+1})$ for each $n\geq 0$. Writing $I_n:= [q_{n-1}',q_{n}']$ for each $n\geq 1$, it follows that
    \begin{align*}
        \sup_{q\in I_n} \frac{\log\lambda_{s+1}(q)}{q} \leq \sup_{q\in I_n} \frac{P(q)}{q}  = \frac{P(q_{n-1})}{q_{n-1}} = \frac{\log Q_{n}}{\log Q_{n} + \log E_{n-1}} \leq 1 - \frac{1}{1+\limsup_{j\rightarrow\infty} \log Q_{j}/E_{j-1}}\,,
    \end{align*}
    and by Mahler's duality, we infer
    \[
        \pu:=\inf_{q\geq 0} \frac{\log\lambda_{1}^*(q)}{q} \geq -1 + \frac{1}{1+\limsup_{n\rightarrow\infty} \log Q_{n}/E_{n-1}}\,.
    \]
    Inequality \eqref{eq: lem Alladi-Rob general: eq1} is a direct consequence of the above combined with \eqref{eq proof: lem alladi-Rob inter 1}.
\end{proof}

We can deduce the following corollary from Theorem~\ref{lem: Alladi-Rob general}.

\begin{corollary} \label{cor lem: Alladi-Rob general}
  Let $\theta\in\R\setminus\{0\}$. Suppose there exist a sequence of integer points $(p_n,q_n)_{n\geq 0}$ and two unbounded increasing sequences of positive real numbers $(E_n)_{n\geq 0}$ and $(Q_n)_{n\geq 0}$ such that, for each integer $n$, the points $(p_n,q_n)$ and $(p_{n+1},q_{n+1})$ are linearly independent and satisfy
  \begin{align*}
    |q_n| \leq Q_n \AND |q_n\theta-p_n|\leq E_n^{-1}.
  \end{align*}
  Then $\theta$ is irrational and its irrationality exponent $\mu(\theta)$ satisfies
  \begin{align}
  \label{eq:mu comme limsup}
    \mu(\theta) \leq 1+\limsup_{n\rightarrow\infty}\frac{\log Q_{n+1}}{\log E_{n-1}}\,.
  \end{align}
\end{corollary}

\begin{proof}
  Set $\btheta:=(1,\theta)$ and for each $n\geq0$ let $M_n$ denote the matrix whose rows are $(-p_n,q_n)$ and $(-p_{n+1},q_{n+1})$. Then $M_n$ satisfies the hypotheses of Theorem~\ref{lem: Alladi-Rob general} with $Q_n$ replaced by $Q_{n+1}$. Theorem~\ref{lem: Alladi-Rob general} yields \eqref{eq:mu comme limsup} by noticing that $\mu(\theta) = \lambda(\btheta)+1$.
\end{proof}

\begin{remark} \label{last remark}
    If $E_n = E^{n(1+o(1))}$ and $Q_n= Q^{n(1+o(1))}$ for given real numbers $E, Q > 1$, then we find that $\theta$ is irrational with
    \begin{align*}
        \mu(\theta) \leq 1+\frac{\log Q}{\log E}\,,
    \end{align*}
    which is exactly \cite[Lemma 3]{A-R} in the case $K=\Q$.
\end{remark}

Our next result is a dual version of Theorem~\ref{lem: Alladi-Rob general}. We will only give a basic proof, although it is also possible to give a proof using parametric of numbers.

\begin{theorem}
    \label{lem: Alladi-Rob general dual}
    Let $s\geq 1$ be an integer and let $\btheta:=(\theta_0,\dots,\theta_s)\in\R^{s+1}$ with $\theta_0=1$. Suppose that there exist a sequence of matrices $(M_n)_{n\geq 0}$ in $\Mat_{s+1}(\Z)\cap \GL_{s+1}(\Q)$ and two unbounded increasing sequences
    $(Q_n)_{n\geq 0}$ and $(E_n)_{n\geq 0}$ of real numbers $\geq 1$ such that, for each $n$ and each row $\bx=(x_0,\dots,x_s)$ of $M_n$, we have
    \begin{align*}
        |x_0|\leq Q_n \AND \max_{1\leq i\leq s}|x_0\theta_i - x_i\theta_0| \leq E_n^{-1}\,.
    \end{align*}
    Then $\theta_0,\dots,\theta_s$ are linearly independent over $\Q$ and
    \begin{align}
    \label{eq: lem Alladi-Rob generalm dual: eq1}
     \omega(\btheta) \leq \limsup_{n\rightarrow\infty}\frac{\log Q_{n}}{\log E_{n-1}}\,.
    \end{align}
    Moreover, if there are positive constants $a,b,\alpha,\beta$ with $\alpha,\beta > 1$ such that $Q_n=a\alpha^n$ and $E_n= b^{-1}\beta^{n}$ for each $n\geq 0$, then for each integer point $(y_0,\dots,y_s)\in\Z^{s+1}$ with $Y:=\sum_{i=1}^{s}|y_i| \geq 1/(2b)$, we have
    \begin{align}
    \label{eq: lem Alladi-Rob general dual: eq2}
         |y_0\theta_0+\cdots + y_s\theta_s| \geq \frac{1}{cY^{\omega}},
    \end{align}
    where $\omega = \log \alpha /\log \beta$ and $c = 2a\alpha (2b)^\omega$.
\end{theorem}

\begin{proof}
    Let $\by=(y_0,\dots,y_s)\in\Z^{s+1}$ with $2\sum_{i=1}^{s}|y_i| \geq E_0$ and let $n\geq 1$ be such that
    \begin{align*}
        E_{n-1} \leq 2\sum_{i=1}^{s}|y_i| \leq E_{n}.
    \end{align*}
    Since $\det M_n \neq 0$, there exists a row $\bx=(\bx_0,\dots,\bx_s)$ of $M_n$ such that $\bx\cdot \by$ is a non-zero integer. We deduce that
    \begin{align*}
        \frac{1}{2} \geq \sum_{i=1}^{s}|y_i|E_n^{-1} \geq \Big|  \sum_{i=0}^{s}y_i(x_0\theta_i-x_i\theta_0) \Big| \geq |\theta_0||\by\cdot\bx| - |x_0||\by\cdot\btheta| \geq 1 - Q_n|\by\cdot\btheta|,
    \end{align*}
    hence the lower bound $2Q_n|\by\cdot\btheta| \geq 1$. In particular, the coordinates of $\btheta$ are linearly independent over $\Q$. We get \eqref{eq: lem Alladi-Rob generalm dual: eq1} by combining the above with the inequality
    \begin{align*}
        Q_n = E_{n-1}^{\log Q_n / \log E_{n-1}} \leq \Big(2\sum_{i=1}^{s}|y_i|\Big)^{\log Q_n / \log E_{n-1}} \leq \Big(2s \norm{\by}\Big)^{\log Q_n / \log E_{n-1}}.
    \end{align*}
    Suppose now that $Q_j=a\alpha^j$ and $E_j= b^{-1}\beta^{j}$ for each $j\geq 0$. Then, we write instead
    \[
        Q_n = a\alpha^n = a\alpha \beta^{(n-1)\log \alpha / \log \beta} \leq a\alpha\Big(2b\sum_{i=1}^{s}|y_i| \Big)^{\log \alpha / \log \beta},
    \]
    and \eqref{eq: lem Alladi-Rob general dual: eq2} follows by combining the above with $2Q_n|\by\cdot\btheta| \geq 1$.
\end{proof}

\section{Effective Poincar\'e-Perron theorem}\label{section: Poincare-Perron thm}

The main result of this section is Theorem~\ref{thm: poincare-Perron effectif}, which implies Theorem~\ref{thm: poincare-Perron effectif application} below. As a consequence, the irrationality measures provided by Theorem~\ref{main} in the case $\alpha\neq 0$ can be made effective (see Remark~\ref{remark: remark effective measures}). Note that in the binomial case, $f(\beta)$ is an algebraic number and Roth's theorem yields $\mu(\beta)=2$, but this irrationality measure is not effective. Then Corollary~\ref{binom} becomes of interest since the irrationality measures that it provides can be made effective.

\begin{theorem}\label{thm: poincare-Perron effectif application}
    Let $\alpha,\beta,\gamma,\delta\in \Q$ with $|\beta|>|\alpha|>0$, $\gamma \ge -1$, $\delta\notin \alpha \N$ and $-(\alpha\gamma+\delta)\notin \alpha \N$. Consider the Poincar\'e-type recurrence \eqref{recurrence}, which can be rewritten as
    \begin{align}\label{eq:PP effectif application:eq1}
      X_{n+1} + a_nX_n + b_nX_{n-1} = 0 \qquad (n\geq 1),
    \end{align}
    where $a_n = -(\beta-B_n)/A_n$, $b_n = C_n/A_n$, and $A_n$, $B_n$, $C_n$ are defined as in Lemma~\ref{rec rel}. Denote by $\lambda_1, \lambda_2$ the roots $2\beta-\alpha \pm 2\sqrt{\beta^2-\beta\alpha}$ of the characteristic polynomial
    \[
        P(X) = X^2-2(2\beta-\alpha)X+\alpha^2.
    \]
    Write $\rho_i:=|\lambda_i|$ for $i=1,2$ and suppose $\rho_1 < \rho_2$. Let $N$ be the smallest positive integer such that
    \begin{align*}
        \frac{2}{|\lambda_2-\lambda_1|}\Big( 2|b_n-\lambda_1\lambda_2|+(\rho_1+\rho_2)|a_n+\lambda_1+\lambda_2|\Big) \leq \rho_2-\rho_1.
    \end{align*}
    Let $(X_n)_{n\geq0}$ be a non-zero solution of \eqref{eq:PP effectif application:eq1}. For each $n\geq N$, define $i_n$ as the largest index $i\in\{1,2\}$ satisfying
    \begin{align*}
        |X_{n+1}-\lambda_iX_n| = \min_{j=1,2}\big\{|X_{n+1}-\lambda_jX_n|\big\}.
    \end{align*}
     Then $(i_n)_{n\geq N}$ is non-decreasing and tends to an index $i\in\{1,2\}$. Given $N_2\geq N$ such that $i_{n} = i$ for each $n\geq N_2$, there exists a positive constant $C$, explicit in function of $\alpha$, $\beta$, $\gamma$, $\delta$, $X_0$, $X_1$ and $N_2$, such that, for each $n\geq 1$, we have
    \begin{align}\label{eq:PP effectif application:eq2}
      |X_n|\leq \frac{C}{\sqrt n}\rho_i^n.
    \end{align}
    Furthermore, there exists an index $N_2$ as above which is explicit in function of $\alpha$, $\beta$, $\gamma$, $\delta$, $X_0$, $X_1$ and $f(\beta)$.
\end{theorem}

The proof of this result is at the end of this section. We now give a more general statement. Let $k$ be a positive integer, and for $i=0,\dots,k-1$ let $a_i\in\C$ and $\ee_i:\N\rightarrow\C$ such that $\ee_i(n)$ tends to $0$ as $n$ tends to infinity. We also suppose that $\ee_0(n)\neq a_0$ for each $n\geq 0$. Consider the Poincar\'e-type recurrence
\begin{align}\label{eq: rec Poincare-Perron}
    g(n+k) + \big(a_{k-1}-\ee_{k-1}(n)\big)g(n+k-1) + \cdots + \big(a_{0}-\ee_{0}(n))g(n) = 0 \qquad (n\geq 0),
\end{align}
and suppose that the roots $\lambda_1,\dots,\lambda_k\in\C$ of its characteristic polynomial $P(X) = X^k+a_{k-1}X^{k-1}+\cdots+a_0$ satisfy $\rho_k > \cdots > \rho_1 \geq 0$, where $\rho_i:=|\lambda_i|$ for $i=1,\dots,k$. Then, the Poincar\'e-Perron theorem states that there is a fundamental system of solutions $(g_1,\dots,g_k)$ of \eqref{eq: rec Poincare-Perron} satisfying
\begin{align*}
  \lim_{n\rightarrow\infty}\frac{g_i(n+1)}{g_i(n)} = \lambda_i \qquad(i=1,\dots,k).
\end{align*}
This is a stronger version of Poincar\'e theorem, which simply states that for each non-zero solution $g$ of \eqref{eq: rec Poincare-Perron}, there exists $i\in\{1,\dots,k\}$ such that $g(n+1)/g(n)$ tends to $\lambda_i$ as $n$ tends to infinity (see \cite{M-T}). When the functions $\ee_i$ converge to $0$ ``fast enough'', it is possible to describe more precisely the asymptotic behavior of a solution. In \cite{Pituk}, M.~Pituk obtained a very nice and general result which implies the following one. Suppose that
\begin{align*}
    \sum_{n=0}^\infty |\ee_i(n)|^2 < \infty \qquad \textrm{for $i=0,\dots,k-1$}.
\end{align*}
Then, there exist $n_0\geq 0$ and a fundamental system of solutions $(g_1,\dots,g_k)$ of \eqref{eq: rec Poincare-Perron} such that, as $n$ tends to infinity, we have
\begin{align*}
  g_i(n) = (1+o(1)) \prod_{m=n_0}^{n-1}\left(\lambda_i + \frac{\lambda_i^{k-1}\ee_{k-1}(m)+\cdots+\lambda_i\ee_1(m)+\ee_0(m)}{P'(\lambda_i)} \right) \qquad(i=1,\dots,k).
\end{align*}
For our purpose, we need an effective version of the above result. However, we could not find a proper reference in the literature. In the following, we denote by $\tau$ the right-shift operator, defined by $(\tau\cdot f)(n) = f(n+1)$ for each function $f:\N\rightarrow \C$ and each $n\geq 0$. For $j=1,\dots,k$, we define $Q_j$ as the polynomial of $\C[X]$ satisfying
\[
    P(X) = (X-\lambda_j)Q_j(X).
\]

The next lemma comes from a classical step in the proof of Poincar\'e Theorem. It will be useful to define the quantities $N$, $i(g)$ and $N_1(g,N)$ appearing in Theorem~\ref{thm: poincare-Perron effectif} below.

\begin{lemma}\label{lem: lemma for Poincare-Perron effectif}
    Let $g$ be a non-zero solution of the Poincar\'e-type recurrence \eqref{eq: rec Poincare-Perron} and let $u_1,\dots,u_k : \N\rightarrow\C$ denote the (unique) functions satisfying
    \begin{align}\label{eq proof: poincare-effectif: eq-1}
        \tau^\ell\cdot g = \lambda_1^\ell u_1 + \cdots + \lambda_k^\ell u_k \qquad(\ell=0,\dots,k-1).
    \end{align}
    Then, for $j=1,\dots, k$ and $n\geq 0$, we have $u_j = P'(\lambda_j)^{-1}Q_j(\tau)\cdot g$ and
    \begin{align}
        \label{eq proof: poincare-effectif: eq2}
        u_j(n+1) = \lambda_ju_j(n) + \frac{E(n)}{P'(\lambda_j)}, \qquad\textrm{where } E:=P(\tau)\cdot g.
    \end{align}
    For each $n\geq 0$, define $i_n$ as the largest index $i\in\{1,\dots,k\}$ such that $|u_i(n)| = \max_{1\leq j \leq k} |u_j(n)|=:u(n)$, and $N$ as  the smallest index satisfying
    \begin{align}
        \label{eq:thm poincare:eq0}
        \Big(\frac{1}{|P'(\lambda_i)|} +  \frac{1}{|P'(\lambda_j)|}\Big)\sum_{\ell=0}^{k-1}(\rho_1^\ell+\cdots+ \rho_k^\ell)|\ee_\ell(n)| < \rho_j-\rho_i \qquad (1\leq i < j \leq k)
    \end{align}
    for each $n\geq N$. Then $(i_n)_{n\geq N}$ is non-decreasing and therefore tends to an index $i(g)\in\{1,\dots,k\}$. We define $N_1(g,N)$ as the smallest $n\geq N$ for which $i_n=i(g)$.
\end{lemma}

\begin{proof}
    For each $n\geq 0$, the $k$-tuple $(u_1(n),\dots,u_k(n))$ is the unique solution of the linear system defined by the $k$ equations
    \begin{align*}
         g(n+\ell) =  \lambda_1^{\ell}u_1(n)+\cdots +\lambda_k^{\ell}u_k(n) \qquad (\ell=0,\dots,k-1)
    \end{align*}
    (this system is not singular since the $\lambda_i$ are pairwise distinct). Note that $u(n)>0$ since $g(n)= \cdots = g(n+k-1)=0$ implies that $g=0$ by \eqref{eq: rec Poincare-Perron}.  Let $j\in\{1,\dots,k\}$ and write $Q_j(X) = \sum_{\ell=0}^{k-1}b_\ell X^\ell$. We now prove \eqref{eq proof: poincare-effectif: eq2}. Using \eqref{eq proof: poincare-effectif: eq-1}, we find
    \begin{align*}
        \big(Q_j(\tau)\cdot g\big)(n) = \sum_{\ell=0}^{k-1}b_\ell g(n+\ell) = \sum_{\ell=0}^{k-1}b_\ell \sum_{p=1}^{k} \lambda_p^\ell u_p(n)
        = \sum_{p=1}^{k} u_p(n) \sum_{\ell=0}^{k-1}b_\ell \lambda_p^\ell
        = \sum_{p=1}^{k} u_p(n) Q_j(\lambda_p).
    \end{align*}
    Observe that $Q_j(\lambda_p) =0$ if $p\neq j$, and $Q_j(\lambda_j) = P'(\lambda_j)\neq 0$, so that $Q_j(\tau)\cdot g = P'(\lambda_j)u_j$. We get \eqref{eq proof: poincare-effectif: eq2} by noticing that $E = P(\tau)\cdot g = (\tau-\lambda_j)\circ Q_j(\tau)\cdot g$.

    \medskip

    It remains to prove that $(i_n)_{n\geq N}$ is non-decreasing (the existence of $i(g)$ and $N_1(g,N)$ easily follows). By definition of $E$ and using \eqref{eq: rec Poincare-Perron} we find
    \begin{align*}
         E(n) = g(n+k)+a_{k-1}g(n+k-1)+\cdots+a_0g(n)  = \sum_{\ell=0}^{k-1}\ee_\ell(n)g(n+\ell)  = \sum_{p=1}^{k}u_p(n)\sum_{\ell=0}^{k-1}\lambda_p^\ell\ee_\ell(n). 
    \end{align*}
    Fix $n\geq N$ and for simplicity, write $i=i_n$ and $j=i_{n+1}$. By \eqref{eq proof: poincare-effectif: eq2} we have
    \begin{align*}
        \rho_iu(n) - \frac{|E(n)|}{|P'(\lambda_i)|} \leq |u_i(n+1)| \leq u(n+1) =  |u_j(n+1)| \leq \rho_ju(n) + \frac{|E(n)|}{|P'(\lambda_j)|}
    \end{align*}
    so that
    \begin{align*}
        \rho_i-\rho_j \leq \Big(\frac{1}{|P'(\lambda_i)|} +  \frac{1}{|P'(\lambda_j)|}\Big) \frac{|E(n)|}{u(n)}
        \leq \Big(\frac{1}{|P'(\lambda_i)|} +  \frac{1}{|P'(\lambda_j)|}\Big) \sum_{\ell=0}^{k-1}\sum_{p=1}^{k}\rho_p^\ell|\ee_\ell(n)|.
    \end{align*}
    Our hypothesis \eqref{eq:thm poincare:eq0} on $N$ implies $j\geq i$. The sequence $(i_n)_{n\geq N}$ is therefore non-decreasing.
\end{proof}

\begin{theorem}\label{thm: poincare-Perron effectif}
    Consider the recurrence \eqref{eq: rec Poincare-Perron} and suppose furthermore that there exist $\ee, \eta : \N\rightarrow (0,+\infty)$ which tend to $0$ as $n$ tends to infinity and satisfy
    \begin{align}\label{eq: def ee and eta eq0}
        \lim_{n\to\infty} \frac{\ee(n+1)}{\ee(n)} = \lim_{n\to\infty} \frac{\eta(n+1)}{\eta(n)} = 1,
    \end{align}
    as well as
    \begin{align}\label{eq: def ee and eta}
        \max_{1\leq i \leq k}|\ee_i(n)|\leq \ee(n) \AND \max_{1\leq i \leq k}|\ee_i(n+1)-\ee_i(n)| \leq \eta(n)
    \end{align}
    for each $n\geq 0$. Let $g$ be a non-zero solution of \eqref{eq: rec Poincare-Perron} and let $N$, $N_1=N_1(g,N)$ and $i=i(g)$ be as in Lemma~\ref{lem: lemma for Poincare-Perron effectif}. Then, there exist positive constants $C$ and $N_2\geq N_1$ depending on  $\ee$, $\eta$, the roots $\lambda_j$, and $N_1$, such that, for each $n\geq N_2$, we have $g(n)\neq 0$ and
    \begin{align*}
      \Big| \frac{g(n+1)}{g(n)} - \Big(\lambda_i + \frac{\lambda_i^{k-1}\ee_{k-1}(n)+\cdots+\lambda_i\ee_1(n)+\ee_0(n)}{P'(\lambda_i)} \Big)  \Big| \leq C\times
      \left\{\begin{array}{ll}
        \displaystyle \big(\ee_0(n)+\ee(n)^2\big)^2 & \textrm{if $\lambda_i=0$} \\
        \displaystyle \eta(n)+\ee(n)^2 & \textrm{if $\lambda_i\neq 0$}.
      \end{array}
      \right.
    \end{align*}
    Moreover, if for $i=0,\dots,k-1$, there exist $\alpha_i,\beta_i\in\C[X]$ with $\beta_i\neq 0$ such that $\ee_i = \alpha_i/\beta_i$, and if \eqref{eq: def ee and eta} are equalities, then $C$ and $N_2$ can be computed explicitly in function of the coefficients of the $\alpha_i,\beta_i$, the roots $\lambda_j$, and $N_1$ only.
\end{theorem}

\textbf{Remarks.} The hypotheses \eqref{eq: def ee and eta eq0} are automatically fulfilled if $\ee(n)$ and $\eta(n)$ are non-zero rational functions in $n$ (this is the situation that we are interested in). Also note that, contrary to the result of Pituk \cite{Pituk}, our theorem covers cases such as $\ee(n) = n^{-1/2}$ or $\ee(n) = 1/\log(n+1)$.

\begin{proof}
     Let $g$ be a non-zero solution of \eqref{eq: rec Poincare-Perron}. We keep the notation of Lemma~\ref{lem: lemma for Poincare-Perron effectif} for the functions $u_1,\dots,u_k$, $u = \max_{1\leq j\leq k}|u_j|$, the function $E$ and the quantities $N$, $N_1 = N_1(g,N)$ and $i = i(g)$ associated with $g$. Recall that
    \begin{align*}
         E(n) = \sum_{p=1}^{k}u_p(n)\sum_{\ell=0}^{k-1}\lambda_p^\ell\ee_\ell(n) = \GrO\big(u(n)\ee(n)\big).
    \end{align*}

    Unless stated otherwise, all the constants (including the implicit constants coming with the big $\GrO$ notation) depend on the functions $\ee$, $\eta$, the roots $\lambda_i$ of $P$, and $N_1$. If for $i=0,\dots,k-1$ we have $\ee_i = \alpha_i/\beta_i$, with $\alpha_i,\beta_i\in\C[X]$ and $\beta_i\neq 0$, then they can be computed explicitly in function of the $\lambda_i$, the coefficients of the $\alpha_i$ and $\beta_i$, and the parameter $N_1$ only. Our goal is to prove that there exists $N_2 \geq N_1$ such that, for each $n\geq N_2$ we have $g(n)\neq 0$ and
    \begin{align}\label{eq proof: poincare-effectif: first step}
        \frac{g(n+1)}{g(n)} =  \lambda_i + \frac{\sum_{\ell=0}^{k-1} \lambda_i^\ell\ee_\ell(n)}{P'(\lambda_i)} +
        \left\{\begin{array}{ll}
                 \GrO\big(\ee_0(n)^2 \big) + \GrO\big(|\ee_0(n)|\ee(n)^2\big)+\GrO\big(\ee(n)^4\big) & \textrm{if $\lambda_i = 0$}\\
                 \GrO\big( \eta(n) + \ee(n)^2 \big) & \textrm{if $\lambda_i\neq 0$}.
               \end{array}
        \right.
    \end{align}
    Fix $j\in\{1,\dots,k\}$ with $j\neq i$ and $n\geq N_1$. Recall that $u(n) = |u_i(n)| \geq |u_j(n)|$.

    \medskip

    \textbf{Case $\lambda_i=0$.} By \eqref{eq proof: poincare-effectif: eq2} we have $u_i(n+1) = \GrO\big(u(n)\ee(n)\big)$, and therefore $|u_j(n+1)| \leq |u_i(n+1)| =  \GrO\big(u(n)\ee(n)\big)$. Eq. \eqref{eq proof: poincare-effectif: eq2} also yields
    \begin{align*}
        u_j(n) = \lambda_j^{-1}u_j(n+1)  + \GrO\big(u(n)\ee(n)\big) =  \GrO\big(u(n)\ee(n)\big),
    \end{align*}
    so that,
    \begin{align*}
        E(n) = u_i(n)\big(\lambda_i^{k-1}\ee_{k-1}(n)+\cdots+\ee_0(n)\big)+\GrO\big(u(n)\ee(n)^2\big) = u_i(n)\ee_0(n)+\GrO\big(u(n)\ee(n)^2\big).
    \end{align*}
    Moreover, we obtain the more precise estimate $|u_j(n+1)| = \GrO\big(u(n+1)\ee(n+1)\big) =  \GrO\big(u(n)\ee(n)\ee(n+1)\big)$. Using once again Eq. \eqref{eq proof: poincare-effectif: eq2}, we get
    \begin{align*}
        \lambda_j\frac{u_j(n)}{u_i(n)} = \frac{u_j(n+1)}{u_i(n)} - \frac{E(n)}{P'(\lambda_j)u_i(n)} = -\frac{\ee_0(n)}{P'(\lambda_j)} +\GrO\big(\ee(n)^2\big).
    \end{align*}
    Thus, using the above estimate and the identity $1/P'(\lambda_1)+\cdots+1/P'(\lambda_k) = 0$ (which can be proven by studying the partial fraction decomposition of $1/\big((x-\lambda_2)\dots(x-\lambda_k)\big)$ evaluated at $x=\lambda_1$), we find
    \[
        \frac{g(n)}{u_i(n)}=\sum_{j=1}^{k}\frac{u_j(n)}{u_i(n)} = 1+\GrO(|\ee_0(n)|\big)+\GrO\big(\ee(n)^2\big)
    \]
    and
    \[
        \frac{g(n+1)}{u_i(n)} = \sum_{j=1}^k\lambda_j\frac{u_j(n)}{u_i(n)} = \sum_{j\neq i}\lambda_j\frac{u_j(n)}{u_i(n)} = -\sum_{j\neq i}\frac{\ee_0(n)}{P'(\lambda_j)} +\GrO\big(\ee(n)^2\big) = \frac{\ee_0(n)}{P'(\lambda_i)}+\GrO\big(\ee(n)^2\big).
    \]
    The case $\lambda_i=0$ of \eqref{eq proof: poincare-effectif: first step} follows.

    \medskip

    \textbf{Case $\lambda_i\neq 0$.} We claim that
    \begin{align}\label{eq proof: poincare-effectif: eq3}
        \frac{u_j(n+1)}{u_i(n+1)} = \frac{\lambda_j}{\lambda_i}\cdot\frac{u_j(n)}{u_i(n)} + \frac{\kappa(n)}{\lambda_iP'(\lambda_j)} + \GrO\big(\ee(n)^2\big) \qquad \textrm{where }  \kappa(n):= \sum_{\ell=0}^{k-1}\lambda_i^\ell\ee_\ell(n),
    \end{align}
    and
    \begin{align}\label{eq proof: poincare-effectif: eq4}
        \frac{u_j(n+1)}{u_i(n+1)} =  \frac{\kappa(n)}{(\lambda_i-\lambda_j)P'(\lambda_j)} + \GrO\big(\eta(n)+\ee(n)^2 \big) = \frac{L_j\kappa(n)}{\lambda_iP'(\lambda_j)} + \GrO\big(\eta(n)+\ee(n)^2 \big), \qquad \textrm{where } L_j:= \frac{\lambda_i}{\lambda_i-\lambda_j}.
    \end{align}
    Let us show that \eqref{eq proof: poincare-effectif: eq4} (for $j\in\{1,\dots,k\}$ with $j\neq i$) implies \eqref{eq proof: poincare-effectif: first step}. Since $g(n) = u_i(n)\big(1+\GrO\big(\ee(n)\big)\big)$, there exists $N_2\geq N_1$ such that $g(n)\neq 0$ for each $n\geq N_2$. Given $n\geq N_2$ and writing $g(n) = u_i(n)+\sum_{j\neq i}u_j(n)$ and $g(n+1)=\sum_{\ell=1}^{k}\lambda_\ell u_\ell(n)$, we then get
    \begin{align*}
      \frac{g(n+1)}{g(n)} = \Big(\lambda_i+\sum_{j\neq i}\lambda_j \frac{u_j(n)}{u_i(n)}\Big)\Big(1-\sum_{j\neq i}\frac{u_j(n)}{u_i(n)}+\GrO\big(\ee(n)^2\big) \Big)
      & = \lambda_i+\sum_{j\neq i}(\lambda_j-\lambda_i) \frac{u_j(n)}{u_i(n)}+\GrO\big(\ee(n)^2\big) \\
      & = \lambda_i-\sum_{j\neq i} \frac{\kappa(n)}{P'(\lambda_j)}+\GrO\big(\eta(n)+\ee(n)^2 \big),
    \end{align*}
    hence \eqref{eq proof: poincare-effectif: first step} (since $1/P'(\lambda_i) = -\sum_{j\neq i} 1/P'(\lambda_j)$).

    \medskip

    Set $q(n):= q_j(n) := u_j(n)/u_i(n)$ and $t(n):= q(n) - L_j\kappa(n)/(\lambda_iP'(\lambda_j))$. Eq. \eqref{eq proof: poincare-effectif: eq4} is equivalent to $t(n)=\GrO(\eta(n)+\ee(n)^2 \big)$. Similarly, we claim that \eqref{eq proof: poincare-effectif: eq3} is implied by $v(n)/u(n) = \GrO(\ee(n))$, where
    \[
         v(n) := \max_{\ell\neq i}\{|u_\ell(n)|\}.
    \]
    Indeed, Eq. \eqref{eq proof: poincare-effectif: eq2} gives $u_{i}(n+1)/u_i(n) = \lambda_i + \GrO\big(\ee(n)\big)$ and
    \begin{align*}
       \frac{u_j(n+1)}{u_i(n)} = \lambda_j\frac{u_j(n)}{u_i(n)}  + \frac{E(n)}{P'(\lambda_j)u_i(n)} = \lambda_j\frac{u_j(n)}{u_i(n)}  + \frac{\kappa(n)}{P'(\lambda_j)} + \GrO\Big(\frac{v(n)}{u(n)}\ee(n)\Big),
    \end{align*}
    hence
    \begin{align}\label{eq proof: poincare-effectif: eq6}
        q(n+1) = \frac{\lambda_j}{\lambda_i}q(n) + \frac{\kappa(n)}{\lambda_iP'(\lambda_j)} + \GrO\Big(\frac{v(n)}{u(n)}\ee(n)\Big) + \GrO(\ee(n)^2) = \frac{\lambda_j}{\lambda_i}q(n) + \GrO\big(\ee(n)\big).
    \end{align}
    This proves our claim. Note that if \eqref{eq proof: poincare-effectif: eq3} holds, then
    \begin{align}\label{eq proof: poincare-effectif: eq5}
        t(n+1) = \frac{\lambda_j}{\lambda_i}t(n) + \frac{\kappa(n)}{\lambda_iP'(\lambda_j)}\Big(1+\frac{\lambda_j}{\lambda_i}L_j\Big) - \frac{L_j}{\lambda_iP'(\lambda_j)}\kappa(n+1) + \GrO\big(\ee(n)^2\big) =  \frac{\lambda_j}{\lambda_i}t(n) +  \GrO\big(\eta(n)+\ee(n)^2 \big)
    \end{align}
    since $1+ L_j\lambda_j/\lambda_i = L_j$ and $\kappa(n+1)-\kappa(n) = \GrO\big(\eta(n)\big)$. To prove that $v(n)/u(n) = \GrO\big(\ee(n)\big)$ and $t(n) = \GrO\big(\eta(n)+\ee(n)^2\big)$, we distinguish between two cases.

    \medskip

    \textbf{Case $i<j$.} Then, by \eqref{eq proof: poincare-effectif: eq6} there exists a constant $c_1$ such that $|q(n)|\leq r|q(n+1)|+c_1\ee(n)$, where $r = |\lambda_i/\lambda_j| < 1$. Fix $s$ with $r < s < 1$. By \eqref{eq: def ee and eta eq0}, we can suppose $N_2$ large enough so that, for each $m\geq n$, we have
    \begin{align*}
        r^{m+1}\ee(m+1) \leq sr^m\ee(m).
    \end{align*}
    Then, by induction and since $r^m|q(n+m)|\leq r^m$ tends to $0$ as $k$ tends to infinity, we get
    \begin{align*}
        |q_j(n)| \leq  c_1\big(\ee(n) + r\ee(n+1) + r^2\ee(n+2)+ \cdots \big) = \GrO\Big(\ee(n)\sum_{m=0}^{\infty} s^m\Big) = \GrO\big(\ee(n)\big).
    \end{align*}
    It implies that $\max_{j\neq i} |q_j(n)| = v(n)/u(n) = \GrO\big(\ee(n)\big)$, so \eqref{eq proof: poincare-effectif: eq3} holds. Similarly, by \eqref{eq proof: poincare-effectif: eq5}, there exists a constant $c_2$ such that
    \begin{align*}
        |t(n)| \leq c_2\sum_{\ell\geq 0}\big(\eta(n+\ell) + \ee(n+\ell)^2\big)r^\ell = \GrO(\eta(n)+\ee(n)^2),
    \end{align*}
    hence \eqref{eq proof: poincare-effectif: eq4}.

    \medskip

    \textbf{Case $i>j$.} Then, by \eqref{eq proof: poincare-effectif: eq6} there exists a constant $c_1$ such that $|q(n+1)|\leq r|q(n)|+c_1\ee(n)$, where $r = |\lambda_j/\lambda_i| < 1$. By induction, we get
    \begin{align*}
        |q(n+1)| \leq r^{n-N_2+1}|q(N_2)|  + c_1\big(\ee(n) + r\ee(n-1) +\cdots + r^{n-N_2}\ee(N_2)\big).
    \end{align*}
    Fix $s$ with $r < s < 1$. Once again, using \eqref{eq: def ee and eta eq0} we can assume, without loss of generality, that $N_2$ is large enough so that, for each $m=1,\dots, n-N_2$, we have
    \begin{align}\label{eq proof: poincare-effectif: eq on N}
        r^m\ee(n-m) \leq sr^{m-1}\ee(n-m+1)
    \end{align}
    as well as $r^m\big(\eta(n-m)+\ee(n-m)^2\big) \leq sr^{m-1}\big(\eta(n-m+1)+\ee(n-m+1)^2\big)$. Eq. \eqref{eq proof: poincare-effectif: eq on N} yields
    \begin{align*}
      \ee(n) + r\ee(n-1) +\cdots + r^{n-N_2}\ee(N_2) = \sum_{m=0}^{n-N_2} r^m\ee(n-m) \leq \ee(n) \sum_{m=0}^{n-N_2} s^m = \GrO\big(\ee(n)\big).
    \end{align*}
    We obtain $q_j(n) = \GrO(r^n)+\GrO\big(\ee(n)\big)= \GrO\big(\ee(n)\big)$, and once again it implies \eqref{eq proof: poincare-effectif: eq3}. Similarly, by \eqref{eq proof: poincare-effectif: eq5}, we find
    \begin{align*}
        |t(n+1)| \leq \GrO\Big(\sum_{m=0}^{n-N_2}\big(\eta(n-m)+\ee(n-m)^2\big)r^m\Big) = \GrO\big(\eta(n)+\ee(n)^2\big).
    \end{align*}
    So \eqref{eq proof: poincare-effectif: eq4} holds too.
\end{proof}

For the sake of completion, we now give a proof of the classical Poincar\'e-Perron theorem.

\begin{theorem}[Poincar\'e-Perron]
    There exists a fundamental system of solutions $(g_1,\dots,g_k)$ of \eqref{eq: rec Poincare-Perron} satisfying
    \begin{align}\label{eq: P-P thm}
        \lim_{n\rightarrow\infty}\frac{g_i(n+1)}{g_i(n)} = \lambda_i\qquad (i=1,\dots,k).
    \end{align}
\end{theorem}

\begin{proof}
    \noindent\textbf{Step $0$.} We first start by giving a short proof of Poincar\'e theorem, namely that $g(n+1)/g(n)$ tends to a root of the characteristic polynomial $P$ for any non-zero solution $g$ of \eqref{eq: rec Poincare-Perron}. We keep the notation of Lemma~\ref{lem: lemma for Poincare-Perron effectif} for the functions $u_1,\dots,u_k$, $u$, $E$ and the quantities $N$, $N_1 = N_1(g,N)$ and $i = i(g)$ associated with such a $g$. Let $j\in\{1,\dots,k\}$ with $j\neq i$ and define $r(n):=|u_j(n)/u_i(n)|\in [0,1]$ for each $n\geq 0$. By \eqref{eq proof: poincare-effectif: eq2}, we have
    \begin{align*}
      r(n+1) = \frac{\rho_jr(n)+o(1)}{\rho_i+o(1)}
    \end{align*}
    as $n$ tends to infinity. If $\rho_i = 0$, then $\rho_j\neq 0$ and the above implies $r(n) = o(1)$ (since $r(n+1)$ is bounded). Assume that $\rho_i\neq 0$ and set $A=\limsup_{n\rightarrow\infty} r(n)$. By the above, we have $A = \rho_j A/\rho_i$. Since $A\leq 1$ and $\rho_j/\rho_i\neq 1$, we must have $A = 0$. So, in both case we find $u_j(n) = o(u_i(n))$ as $n$ tends to infinity. Writing $g(n) = u_1(n)+\cdots + u_k(n)$ and $g(n+1) = \lambda_1u_1(n)+\cdots + \lambda_ku_k(n)$, we deduce that $\lim_{n\rightarrow\infty}g(n+1)/g(n) = \lambda_i$. We denote this root by $\lambda(g):=\lambda_{i(g)}$.

    \medskip

    \noindent\textbf{Step $1$.} We now prove the Poincar\'e-Perron theorem. Since $\rho_1 < \cdots < \rho_k$, if $i(h)< i(g)$ for two solutions $h$, $g$, then $h(n)/g(n)$ tends to $0$ as $n$ tends to infinity. In particular, solutions $(g_1,\dots,g_k)$ satisfying \eqref{eq: P-P thm} are necessarily linearly independent over $\C$. Moreover, for $i=1,\dots,k$, we also find that the set $S_i$ of solutions $g$ such that $i(g)\leq i$ is a subspace of the $k$-dimensional space of solutions. The existence of $(g_1,\dots,g_k)$ as above is equivalent to
    \begin{align}\label{eq proof: equivalence of P-P thm}
        \dim S_i = i\qquad (i=1,\dots,k).
    \end{align}
    Suppose $i\geq 2$ and let $g_1,\dots,g_{i}$ be $i$ linearly independent solutions of \eqref{eq: rec Poincare-Perron}. For $j=1,\dots,m$, we denote by $u_1^{(j)},\dots, u_k^{(j)}$ the functions $u_1,\dots,u_k$ associated to $g_i$ by Lemma~\ref{lem: lemma for Poincare-Perron effectif}. Let $\mu_1,\dots,\mu_{i}\in\C$ (not all zero) and set $g = \mu_1g_1+\cdots +\mu_{i}g_i$. Denote by $(i_n)_{n\geq N}$ and by $u_1,\dots,u_k$ the quantities associated to $g$ by Lemma~\ref{lem: lemma for Poincare-Perron effectif}. Then
    \begin{align*}
        u_j=\mu_1u_j^{(1)}+\cdots + \mu_{i}u_j^{(i)}\qquad(j=1,\dots,k).
    \end{align*}
    By choosing $\mu_1,\dots,\mu_i$ (not all $0$) so that $u_1(N) = \cdots = u_{i-1}(N) = 0$, we get $i_N\geq i$ by definition of $i_N$. Since $(i_n)_{n\geq N}$ is non-decreasing, it yields $i(g)\geq i$, and so $g\notin S_{i-1}$. It shows that $\dim S_{i-1}\leq i-1$. We now prove that given $g, h\in S_i\setminus S_{i-1}$, there exists a non-zero $\alpha\in\C$ such that $h(n)/g(n)$ converges to $\alpha$ as $n$ tends to infinity. In particular, since for $\widehat{h}=h-\alpha g\in S_i$ the ratio $\widehat{h}(n)/g(n)$ tends to $0$, necessarily $\widehat{h} \in S_{i-1}$. As a consequence $\dim S_{i-1} \geq (\dim S_i) -1$. Combined with $\dim S_i\leq i$ and $\dim S_k = k$, it implies \eqref{eq proof: equivalence of P-P thm} (and thus implies the Poincar\'e-Perron theorem). Let $h, g$ be as above and for simplicity, write $\lambda:=\lambda_i$ (note that $\lambda(g) = \lambda(h) = \lambda$). Let $N_2\geq N$ be such that $g(n)\neq 0$ and $h(n)\neq 0$ for each $n\geq N_2$, and for such $n$ define
    \begin{align*}
        q(n):=\frac{h(n)}{g(n)} \AND t(n):= \frac{q(n+1)}{q(n)}-1.
    \end{align*}
    Note that $t(n) = o(1)$ since $g(n+1)/g(n)$ and $h(n+1)/h(n)$ both tend to $\lambda$. We suppose $N_2$ large enough so that $|t(n)| < 1$. The function $q$ is a solution of the Poincar\'e-type recurrence
    \begin{align}\label{eq: rec Poincare-Perron bis}
        q(n+k) + \widetilde{a}_{k-1}(n)q(n+k-1) + \cdots + \widetilde{a}_{0}(n)q(n) = 0 \qquad (n\geq N_2),
    \end{align}
    where for $j=0,\dots,k-1$, we have
    \begin{align*}
        \widetilde{a}_j(n) = \big(a_{j}-\ee_{j}(n)\big)\frac{g(n+j)}{g(n+k)} = \frac{a_j}{\lambda^{k-j}}+o(1)
    \end{align*}
    as $n$ tends to infinity. The characteristic polynomial of the above recurrence is precisely $P(\lambda X)/\lambda^k$ Moreover, since $g$ is solution of \eqref{eq: rec Poincare-Perron}, the solution constant equals to $1$ is solution of \eqref{eq: rec Poincare-Perron bis}, so that
    \begin{align*}
        -1 = \widetilde{a}_{k-1}(n) + \cdots + \widetilde{a}_{0}(n) \qquad (n\geq N_2).
    \end{align*}
    We deduce from \eqref{eq: rec Poincare-Perron bis} and the above that
    \begin{align*}
        t(n+k-1) = \frac{q(n+k)}{q(n+k-1)} - 1 = \sum_{j=0}^{k-1}\widetilde{a}_j(n)\Big(1-\frac{q(n+j)}{q(n+k-1)}\Big)
        & = \sum_{j=0}^{k-2}\widetilde{a}_j(n)\frac{q(n+j)}{q(n+k-1)}\Big(\frac{q(n+k-1)}{q(n+j)}-1\Big).
    \end{align*}
    Now, since $t(n) = o(1)$, for $j=0,\dots,k-2$, we have
    \begin{align*}
        \frac{q(n+k-1)}{q(n+j)}-1 = \prod_{\ell=j}^{k-2} \big(t(n+\ell)+1\big) - 1 = \sum_{\ell=j}^{k-2} (1+o(1))t(n+\ell).
    \end{align*}
    Therefore, the function $t$ satisfies a Poincar\'e-type recurrence of order $k-1$ whose characteristic polynomial is
    \begin{align*}
        Q(X) = X^{k-1} - \sum_{j=0}^{k-2}\frac{a_j}{\lambda^{k-j}}\big(X^j+\cdots+X^{k-2}\big).
    \end{align*}
    A short computation shows that $(X-1)Q(X) = P(\lambda X)/\lambda^k$, so the roots of $Q$ are $\lambda_j/\lambda_i$ with $j\neq i$. Poincar\'e theorem ensures that $t(n+1)/t(n)$ converges to one of these roots. Since $t(n) = o(1)$, such a root has modulus $<1$. So, the series $\sum_{n\geq N_2}t(n)$ converges absolutely. Equivalently, the infinite product $\prod_{n\geq N_2}(1+t(n))$ converges in $\C^*$. Since $t(n)+1 = q(n+1)/q(n)$, we obtain the convergence in $\C^*$ of $q(n) = h(n)/g(n)$. This ends the proof of the Poincar\'e-Perron theorem.
\end{proof}

\begin{proof}[Proof of Theorem \ref{thm: poincare-Perron effectif application}]
  A simple computation gives
    \begin{align*}
        a_n = -2(2\beta-\alpha) + \frac{2\beta-\alpha}{n}+ \GrO\Big(\frac{1}{n^2}\Big) \AND
        b_n  = \alpha^2 - \frac{\alpha^2}{n} + \GrO\Big(\frac{1}{n^2}\Big),
    \end{align*}
  which can be rewritten as
  \begin{align*}
    a_n = -(\lambda_1+\lambda_2) + \frac{\lambda_1+\lambda_2}{2n} + \GrO\Big(\frac{1}{n^2}\Big) \AND
        b_n  = \lambda_1\lambda_2 - \frac{\lambda_1\lambda_2}{n} + \GrO\Big(\frac{1}{n^2}\Big),
  \end{align*}
  where the implicit constants depend on $\alpha$, $\beta$, $\gamma$, $\delta$ only (and can be computed explicitly). Note that the index $N$ and the sequence $(i_n)_{n\geq N}$ are those of Theorem \ref{thm: poincare-Perron effectif}. Therefore, there exist $i,j\in\{1,2\}$ with $i\neq j$ such that, for each $n\geq N_2$, we have
  \begin{align*}
     \frac{X_{n+1}}{X_n} = \lambda_i + \frac{-(\lambda_i+\lambda_j)\lambda_i+2\lambda_i\lambda_j}{\lambda_i-\lambda_j}\cdot\frac{1}{2n} + \GrO\Big(\frac{1}{n^2}\Big)  =
     \lambda_i -\lambda_i\frac{1}{2n} + \GrO\Big(\frac{1}{n^2}\Big),
  \end{align*}
  from which we deduce
  \begin{align*}
    |X_n| \leq \rho_i^n\prod_{k=1}^{n}\Big( 1 - \frac{1}{2n} + \GrO\Big(\frac{1}{n^2} \Big)\Big) = \GrO\Big(\frac{\rho_i^n}{\sqrt n} \Big),
  \end{align*}
  where the implicit constant depends on $\alpha$, $\beta$, $\gamma$, $\delta$, $X_0$, $X_1$ and $N_2$. It remains to prove that $N_2$ can be explicitly chosen in function of $\alpha$, $\beta$, $\gamma$, $\delta$, $X_0$, $X_1$ and $f(\beta)$. Let $Y=(Y_n)_{n\geq 0}$ be a non-zero solution of \eqref{eq:PP effectif application:eq1} such that $Y_{n+1}/Y_n$ tends to $\lambda_1$ (up to a multiplicative constant, this sequence is $(R_n)_{n\geq 0}$ defined as in Lemma~\ref{rec rel}). Then, since the associated sequence $(i_n(Y))_{n\geq N}$ is non-decreasing and tends to $1$, we have $i_n = 1$ for each $n\geq N$, so that $N_1(Y,N) = N$. Similarly, we have $N_1(Z,N)=N$ for the solution $Z=(Z_n)_{n\geq 0}$ defined by the condition
  \[
    X_N = 1+1 \AND X_{N+1} = \lambda_1 + \lambda_2,
  \]
  since the corresponding non-decreasing sequence $(i_n(Z))_{n\geq N}$ satisfies $i_N = 2$. By the above, we can assume that $X_{n+1}/X_n$ tends to $\lambda_2$ (otherwise we can take $N_2 = N_1 = N$). Then, writing $(X_n)_{n\geq 0} = aY + bZ$ (where $a,b$ depend on the two initial values of  $(X_n)_{n\geq 0}$, $Y$ and $Z$), we can deduce an index $N_2\geq N$, explicit in function of $\alpha$, $\beta$, $\gamma$, $\delta$, $a$, $b$, $Y_0$, $Y_1$, $Z_0$ and $Z_1$, for which $|X_{N_2+1}-\lambda_1X_{N_2}| > |X_{N_2+1}-\lambda_2X_{N_2}|$. Finally, note that $Z_0$ and $Z_1$ depends only on $\alpha$, $\beta$, $\gamma$, $\delta$, and $Y_1$ depends only on $Y_0$, $\alpha$, $\beta$, $\gamma$, $\delta$ and $f(\beta)$ (since $Y$ is proportional to $(R_n)_{n\geq 0}$).
\end{proof}

\section{Examples}\label{example}

In this section, we compute the irrationality measures for some cubic roots and compare them to previous results.

\begin{lemma} \label{upper W}
    We keep the notation of Corollary $\ref{binom}$. Let $\beta\in \Q$ with $|\beta|>1$. Then we have
    \[
        \Delta(1/3,\beta)\le \dfrac{3\sqrt{3}\cdot {\rm{den}}(\beta)}{2}.
    \]
\end{lemma}

\begin{proof}
    Let $(P_{n,0}(z),P_{n,1}(z))$ be the weight $n$ Pad\'{e} approximants of $1/z(1-1/z)^{1/3}$ defined as in Corollary $\ref{binomial pade}$ for $\omega=1/3$. Notice that $(z^nP_{n,0}(1/z),z^{n-1}P_{n,1}(1/z))$ is a weight $(n,n-1)$ Pad\'{e} approximant of $(1,(1-z)^{1/3})$. By \cite[Lemma $3.3$]{Rick}, we have
    \begin{align*}
        &z^nP_{n,0}(1/z)
        =(-1)^n\sum_{k=0}^n\left[\sum_{j=0}^{k}(-1)^{k-j}\binom{n+1/3}{j}\binom{2n-1-j}{n-j}\binom{n-j}{k-j}\right]z^{k},\\
        &z^{n-1}P_{n,1}(1/z)=(-1)^{n+1}\sum_{k=0}^{n-1}\binom{n+1/3}{k}\binom{2n-k-1}{n-1-k}z^k.
    \end{align*}
    Set
    \[
        G(n):={\rm{GCD}}\left(3^{\lfloor 3n/2\rfloor}\binom{n+1/3}{j}\binom{2n-j-1}{n-j}\right)_{0\le j \le n},
    \]
    and $G_n:={\rm{GCD}}(G(n),G(n-1))$. Then Bennett shows in \cite[Lemma $3.2$]{BennettSimul} that $G_n>(1/5563)\cdot 2^n$.
    This implies that $G_n(1/3)>(1/5563)\cdot 2^n$. The lemma follows from the above and the definition of $\Delta(1/3,\beta)$.
\end{proof}

\begin{example} \label{20}
    Let $\omega=1/3$ and $\beta=9$. Then, by Lemma $\ref{upper W}$, we have $\Delta(1/3,9)\le 3\sqrt{3}/2$ and $\rho_2(1,\beta)=17+12\sqrt{2}$. Corollary $\ref{binom}$ yields
    \begin{align*}
        \mu(\sqrt[3]{3})\le 1+\dfrac{\log(17+12\sqrt{2})+\log\left(3\sqrt{3}/2\right)}{\log(17+12\sqrt{2})-\log\left(3\sqrt{3}/2\right)}=2.7428036524\cdots.
    \end{align*}
    Note that Bennett obtained $\mu(\sqrt[3]{3})\le 2.76$ in \cite{BennettAus}.
\end{example}

In the tables below, the number $\theta$ can be written as the product of $\beta^{-1}(1-1/\beta)^{1/3}$ by a non-zero rational number. The exponent $\mu$ is the irrationality measure for $\beta^{-1}(1-1/\beta)^{1/3}$ (and thus $\theta$) obtained by combining Lemma $\ref{upper W}$ and Corollary $\ref{binom}$ as in Example~\ref{20}. In the last column, we put the irrationality measure obtained by Bennett in \cite{BennettAus}.

\begin{figure}[H]
    {\renewcommand{\arraystretch}{1.2}
        \centering

        \begin{minipage}{0.5\textwidth}
          \[
            \begin{array}{|c|| c | c | c |}
            \hline
            \theta            &  \beta                     & \mu                    & \textbf{results in \cite{BennettAus}}   \\ \hline
            \sqrt[3]{3}       &       9                    & 2.74\cdots             &  2.76\cdots       \\ \hline
            \sqrt[3]{6}       &       467^3/5              & 2.32\cdots             &  2.35\cdots       \\ \hline
            \sqrt[3]{15}      &       25                   & 2.52\cdots             &  2.54\cdots       \\ \hline
            \sqrt[3]{17}      &       18^3                 & 2.20\cdots             &  2.22\cdots       \\ \hline
            \sqrt[3]{19}      &       -8^3                 & 2.28\cdots             &  2.30\cdots       \\ \hline
            \sqrt[3]{20}      &       -19^3                & 2.20\cdots             &  2.23\cdots       \\ \hline
            \sqrt[3]{26}      &       3^3                  & 2.51\cdots             &  2.53\cdots       \\ \hline
            \sqrt[3]{28}      &       -3^3                 & 2.50\cdots             &  2.52\cdots       \\ \hline
            \sqrt[3]{30}      &       -9                   & 2.71\cdots             &  2.72\cdots       \\ \hline
            \end{array}
        \]
        \end{minipage}
        \hfillx
        \begin{minipage}{0.5\textwidth}
        \[
            \begin{array}{|c|| c | c | c |}
            \hline
            \theta            &  \beta                     & \mu                    & \textbf{results in \cite{BennettAus}}   \\ \hline
            \sqrt[3]{37}      &       10^3                 & 2.26\cdots             &  2.27\cdots       \\ \hline
            \sqrt[3]{42}      &       49                   & 2.44\cdots             &  2.46\cdots       \\ \hline
            \sqrt[3]{43}      &       -7^3                 & 2.30\cdots             &  2.32\cdots       \\ \hline
            \sqrt[3]{62}      &       32                   & 2.49\cdots             &  2.50\cdots       \\ \hline
            \sqrt[3]{63}      &       4^3                  & 2.41\cdots             &  2.43\cdots       \\ \hline
            \sqrt[3]{65}      &       -4^3                 & 2.41\cdots             &  2.43\cdots       \\ \hline
            \sqrt[3]{66}      &       -32                  & 2.48\cdots             &  2.50\cdots       \\ \hline
            \sqrt[3]{83}      &       -(253)^3/19          & 2.69\cdots             &  2.72\cdots       \\ \hline
            \sqrt[3]{91}      &       9^3                  & 2.27\cdots             &  2.29\cdots       \\ \hline
            \end{array}
        \]
        \end{minipage}

    }
   \caption{Effective irrationality measures for some cubic roots}
   \label{intro:table1}
\end{figure}


\section{Algebraic binomial case}\label{bin}

In this section, we show a $p$-adic and algebraic version of Theorem $\ref{main}$ for binomial functions. We will adapt it to study the $S$-unit equation in a forthcoming work.

\medskip

Let $K$ be a number field. We denote the set of places of $K$ by ${{\mathfrak{M}}}_K$ (respectively by ${\mathfrak{M}}^{\infty}_K$ for archimedean places, by ${{\mathfrak{M}}}^{f}_K$ for finite places).
Given $v\in {{\mathfrak{M}}}_K$, we denote by $K_v$ the completion of $K$ with respect to $v$. 
We define the normalized absolute value $| \cdot|_v$ as follows~:
\begin{align*}
    |p|_v &:=p^{-\tfrac{[K_v:\Q_p]}{[K:\Q]}} \qquad  \text{if }  v\in{{\mathfrak{M}}}^{f}_K  \text{ and }  v\mid p,\\
    |x|_v & :=|\sigma_v( x)|^{\tfrac{[K_v:\R]}{[K:\Q]}} \qquad  \text{if }  v\in {{\mathfrak{M}}}^{\infty}_K,
\end{align*}
where $p$ is a rational prime and $\sigma_v$ the embedding $K\hookrightarrow \C$ corresponding to $v$.

\medskip

Let $\beta$ be an algebraic number. We denote the $v$-adic absolute Weil height of $\beta$ by $H_v(\beta)=\max(1,|\beta|_v)$, the absolute Weil height of $\beta$ by
\[
    H(\beta)=\prod_{v\in \mathfrak{M}_K}H_v(\beta).
\]
\begin{lemma} \label{res p}
    Let $K$ be an algebraic number field and $v_0$ a non-archimedean place. Let $\omega\in \Q\setminus \Z$.
    Denote by $p$ be the rational prime under $v_0$. Define $\delta_p(\omega)=\begin{cases} 0 & \ \text{if} \ p\nmid {\rm{den}}(\omega)\\ 1& \ \text{if} \ p\mid {\rm{den}}(\omega)\end{cases}$.
    Let $\beta\in K$. Assume
    \begin{align}\label{beta}
        |\beta|_{v_0}>
        \begin{cases}
            1 & \ \ \text{if} \ \ p\nmid {\rm{den}}(\omega)\\
            |p|^{-p/(p-1)}_{v_0} & \ \ \text{if} \ \ p \mid {\rm{den}}(\omega).
        \end{cases}
    \end{align}
    Let $n$ be a non-negative integer and $(P_{n,0}(z),P_{n,1}(z))$ be the weight $n$ Pad\'{e} approximants of $1/z(1-1/z)^{\omega}$ defined as in Corollary $\ref{binomial pade}$.
    Put $R_n(z)=P_{n,0}(z)\cdot 1/z(1-1/z)^{\omega}-P_{n,1}(z)$.
    Then we have
    \[
        \left|\dfrac{\nu_n(\omega){\rm{den}}(\beta)^n}{G_n(\omega)}R_n(\beta)\right|_{v_0}\le \left(|p|^{-2p\delta_p(\omega)/(p-1)}_{v_0}\left|\dfrac{{\rm{den}}(\beta)}{\beta}\right|_{v_0}\right)^n.
    \]
\end{lemma}
\begin{proof}
    Set
    \[
        \dfrac{\nu_n(\omega)}{G_n(\omega)} P_{n,0}(z)=\sum_{j=0}^na_jz^j\in \Z[z].
    \]
    Since $R_n(z)$ belongs to the ideal $(1/z^{n+1})$ of $\Q[[1/z]]$ and
    \[
        \dfrac{1}{z}\left(1-\dfrac{1}{z}\right)^{\omega}=\sum_{k=0}^{\infty}\dfrac{(-\omega)_k}{k!}\dfrac{1}{z^{k+1}},
    \]
    we have
    \[
        \dfrac{\nu_n(\omega){\rm{den}}(\beta)^n}{G_n(\omega)}R_n(z)=\sum_{k=n}^{\infty}\left(\sum_{j=0}^n\dfrac{a_j\cdot (-\omega)_{k+j}}{(k+j)!}\right)\dfrac{1}{z^{k+1}}.
    \]
    Using the identity above and the strong triangle inequality, we have
    \[
        \left|\dfrac{\nu_n(\omega){\rm{den}}(\beta)^n}{G_n(\omega)}R_n(\beta)\right|_{v_0}\le
        \max_{k\ge n} \left(\left|\dfrac{\nu_{n+k}(\omega)^{-1}{\rm{den}}(\beta)^n}{\beta^{k+1}}\right|_{v_0}\right)
        \le \left(|p|^{-2p\delta_p(\omega)/(p-1)}_{v_0}\left|\dfrac{{\rm{den}}(\beta)}{\beta}\right|_{v_0}\right)^n.
    \]
    Note that the last inequality follows from $(\ref{beta})$.
\end{proof}
\begin{theorem}\label{binomalgebraic}
    Let $\omega\in \Q\setminus\Z$.
    Let $K$ be an algebraic number field and $v_0$ a place of $K$.
    Let $\beta\in K$ with
    \[
        |\beta|_{v_0}>
        \begin{cases}
            1 \ \ &\text{if} \ v_0\mid \infty, \ \text{or} \ v_0\nmid \infty \ \text{and} \ \ p\nmid {\rm{den}}(\omega)\\
            |p|^{-p/(p-1)}_{v_0} \ \ & \text{otherwise},
        \end{cases}
    \]
    where $p$ is the rational prime under $v_0$ when $v_0\nmid \infty$. Define the real number $:$
    \begin{align*}
        &V(\beta):=
        \begin{cases}
            -\dfrac{[K_{v_0}:\R]}{[K:\Q]}\log \rho_1(1,\sigma_{v_0}(\beta))-\log\,\Delta(\omega,\beta)-\sum_{\substack{v\mid\infty \\ v\neq v_0}}\dfrac{[K_{v}:\R]}{[K:\Q]}\log \rho_2(1,\sigma_{v}(\beta)) &\ \ \text{if} \ v_0 \mid \infty\\
            -\log \left(|p|^{-2p\delta_p(\omega)/(p-1)}_{v_0}\left|{\rm{den}}(\beta)/\beta\right|_{v_0}\right)-\log\,\Delta(\omega,\beta)-\sum_{\substack{v\mid\infty}}\dfrac{[K_{v}:\R]}{[K:\Q]}\log \rho_2(1,\sigma_{v}(\beta)) & \ \ \text{if} \ v_0\nmid \infty.
        \end{cases}
    \end{align*}
    Assume $V(\beta)>0$.
    Then for any $0<\varepsilon<V(\beta)$, there exists a constant $H_0=H_0(\varepsilon)>0$ depending on $\varepsilon$ and the given data such that the following property holds.
    For any ${{\boldsymbol{\lambda}:=(\lambda_0,\lambda_1)}} \in K^{2} \setminus \{ \bold{0} \}$ satisfying $H_0\le H({\boldsymbol{\lambda}})$, we have
    \begin{align*}
        \left|\lambda_0(1-1/\beta)^{\omega}-\lambda_1\right|_{v_0}>C(\varepsilon) H_{v_0}({\boldsymbol{\lambda}}) H({\boldsymbol{\lambda}})^{-\mu(\varepsilon)},
    \end{align*}
    where
    \begin{align*}
        &\mu(\varepsilon):=
        \begin{cases}
            \dfrac{-[K_{v_0}:\R](\log \rho_1(1,\sigma_{v_0}(\beta)-\log \rho_2(1,\sigma_{v_0}(\beta))}{[K:\Q](V(\beta)-\varepsilon)}  & \ \ \text{if} \ v_0\mid \infty\\
            \dfrac{\log\left(|p|^{-2p\delta_p(\omega)/(p-1)}_{v_0}\left|{\rm{den}}(\beta)/\beta\right|_{v_0}\right)}{(V(\beta)-\varepsilon)} & \ \ \text{if} \ v_0\nmid \infty,
        \end{cases}\\
        &C{{(\varepsilon)}}:=\exp\left[-{{\left(\dfrac{\log(2)}{V(\beta)-\varepsilon}+1\right)}}\right]\cdot
        \begin{cases}
            \exp\left[ \dfrac{-[K_{v_0}:\R](\log \rho_1(1,\sigma_{v_0}(\beta)-\log \rho_2(1,\sigma_{v_0}(\beta))}{[K:\Q]}\right]  & \ \ \text{if} \ v_0\mid \infty\\
            |p|^{-2p\delta_p(\omega)/(p-1)}_{v_0}\cdot \left|{\rm{den}}(\beta)/\beta\right|_{v_0} & \ \ \text{if} \ v_0\nmid \infty.
        \end{cases}
    \end{align*}
\end{theorem}
\begin{proof}
    Let $n$ be a non-negative integer. Let $(P_{n,0}(z),P_{n,1}(z))$ be the weight $n$ Pad\'{e} approximants of $1/z(1-1/z)^{\omega}$ defined in Corollary $\ref{binomial pade}$.
    We put
    \[
        p_n:=\dfrac{\nu_n(\omega){\rm{den}}(\beta)^nP_{n,0}(\beta)}{G_n(\omega)}, \ \ q_n:=\dfrac{\nu_n(\omega){\rm{den}}(\beta)^nP_{n,1}(\beta)}{G_n(\omega)}, \ \ M_n:=\begin{pmatrix}p_n & q_n\\ p_{n+1} & q_{n+1}\end{pmatrix}.
    \]
    Thanks to Lemma $\ref{lem: Delta_n}$, we have $M_n\in {\rm{GL}}_2(K)$.
    Using Lemma $\ref{apply P-P}$ and Lemma $\ref{res p}$, we get
    \begin{align*}
        &\log\,|p_n|_{v_0}\le
        \begin{cases}
            \dfrac{[K_{v_0}:\R]}{[K:\Q]}\left(\log \rho_2(1,\sigma_{v_0}(\beta))+\log\,\Delta(\omega,\beta)\right)n+o(n) &\ \ \text{if} \ v_0\mid \infty\\
            0 & \ \ \text{if} \ v_0\nmid \infty,
        \end{cases}\\
        &\log\,|p_n\cdot 1/\beta(1-1/\beta)^{\omega}-q_n|_{v_0}\le
        \begin{cases}
            \dfrac{[K_{v_0}:\R]}{[K:\Q]}(\log\rho_1(1,\sigma_{v_0}(\beta))+\log\,\Delta(\omega,\beta))n+o(n) & \ \ \text{if} \ v\mid \infty\\
            n\log \left(|p|^{-2p\delta_p(\omega)/(p-1)}_{v_0}\left|{\rm{den}}(\beta)/\beta\right|_{v_0}\right) & \ \ \text{if} \ v\nmid \infty,
        \end{cases}\\
        &\log\|M_n\|_{v}\le
        \begin{cases}
            \dfrac{[K_{v_0}:\R]}{[K:\Q]}(\log\rho_2(1,\sigma_{v_0}(\beta))+\log\,\Delta(\omega,\beta))n+o(n)   & \ \ \text{if} \ v\mid \infty\\
            0 & \ \ \text{if} \ v\nmid \infty.
        \end{cases}
    \end{align*}
    Applying linear independence criterion \cite[Proposition $3$]{DHK3} for $\theta_1:=1/\beta(1-1/\beta)^{\omega}$ and the family of invertible matrices $(M_n)_{n\geq 0}$, we obtain the assertion.
\end{proof}

{\bf Acknowledgments}

The authors are grateful to Professor Sinnou David for his suggestions. The second author is supported by JSPS Postdoctoral Fellowships No. PE20746 for Research in Japan, together with Research Support Allowance Grant for JSPS Fellows. He is also thankful for the hospitality of College of Science and Technology, Nihon University.

\bibliography{}


\

\begin{scriptsize}
\begin{minipage}[t]{0.5\textwidth}

Po\"{e}ls, Anthony,
\\poels.anthony@nihon-u.ac.jp
\\Department of Mathematics
\\College of Science \& Technology
\\Nihon University
\\Kanda, Chiyoda, Tokyo
\\101-8308, Japan\\\\
\end{minipage}
\begin{minipage}[t]{0.5\textwidth}
Makoto Kawashima,
\\kawashima.makoto@nihon-u.ac.jp
\\Department of Liberal Arts \\and Basic Sciences
\\College of Industrial Engineering
\\Nihon University
\\Izumi-chou, Narashino, Chiba
\\275-8575, Japan\\\\
\end{minipage}

\end{scriptsize}

\end{document}